\documentclass[reqno,11pt]{amsart}
\usepackage{enumitem}
\usepackage{mathtools}
\usepackage{amsmath}
\usepackage{amssymb}
\usepackage{yhmath}
\usepackage{graphicx}
\usepackage{mathrsfs}
\usepackage{bbm}
\usepackage{xcolor}
\usepackage{tikz-cd}
\usepackage[colorinlistoftodos,prependcaption,textsize=tiny,obeyFinal]{todonotes}

\DeclareMathAlphabet{\mathpzc}{OT1}{pzc}{m}{it}

\newtheorem{theorem}{Theorem}[section]

\newtheorem*{claim*}{Claim}

\newtheorem{lemma}[theorem]{Lemma}
\newtheorem{lem}[theorem]{Lemma}

\newtheorem{cor}[theorem]{Corollary}

\newtheorem{prop}[theorem]{Proposition}

\newtheorem{thm}[theorem]{Theorem}

\theoremstyle{definition}
\newtheorem{definition}[theorem]{Definition}

\theoremstyle{remark}

\newtheorem{Rmk}[theorem]{Remark}
\newtheorem{Q}[theorem]{Open problem}

\numberwithin{equation}{section}

\newcommand{\op}{\operatorname}

\newcommand{\bb}{\mathbb}

\newcommand{\Ga}{\Gamma}
\newcommand{\ga}{\gamma}

\newcommand{\la}{\lambda}
\newcommand{\La}{\Lambda}
\newcommand{\ba}{\backslash}

\newcommand{\cal}{\mathcal}
\newcommand{\br}{\mathbb R}
\newcommand{\SO}{\op{SO}}
\newcommand{\SU}{\op{SU}}
\newcommand{\Sp}{\op{Sp}}

\newcommand{\G}{\Gamma}
\newcommand{\m}{\mathsf{m}}

\newcommand{\T}{\op{T}}\newcommand{\F}{\cal{F}}
\renewcommand{\frak}{\mathfrak}
\newcommand{\e}{\epsilon}

\newcommand{\be}{\begin{equation}}
\newcommand{\ee}{\end{equation}}
\newcommand{\BR}{\op{BR}}
\newcommand{\inte}{\op{int}}
\renewcommand{\L}{\mathcal L}
\newcommand{\fa}{\mathfrak a}

\newcommand{\fg}{\mathfrak{g}}
\newcommand{\fp}{\mathfrak{p}}

\renewcommand{\T}{\mathsf{T}}

\newcommand{\fk}{\mathfrak{k}}

\newcommand{\N}{\mathbb N}
\renewcommand{\e}{\varepsilon}
\renewcommand{\epsilon}{\varepsilon}
\newcommand{\pc}{P^{\circ}}

\newcommand{\E}{\mathcal E}
\newcommand{\Om}{\Omega}
\renewcommand{\u}{\mathsf u}
\newcommand{\R}{\cal R}
\renewcommand{\v}{\mathsf v}
\newcommand{\PSL}{\op{PSL}}

\renewcommand{\r}{\mathsf r}
\renewcommand{\d}{\mathsf d}

\newcommand{\tOm}{\tilde{\Omega}}
\newcommand{\injrad}{\operatorname{rad_{\mathrm{inj}}}}

\begin{document}

\title[Invariant measures and rank dichotomy]
{Horospherical invariant measures and a rank dichotomy for Anosov groups}
\author{Or Landesberg}
\address{Mathematics department, Yale university, New Haven, CT 06520}
\email{or.landesberg@yale.edu}
\author{Minju Lee}
\address{Mathematics department, Yale university, New Haven, CT 06520, and Department of Mathematics, University of Chicago, Chicago, IL 60637 (current address)}
\email{minju1@uchicago.edu}
\author{Elon Lindenstrauss}
\address{Mathematics department, Hebrew University, Jerusalem, Israel}
\email{elon@math.huji.ac.il}
\author{Hee Oh}
\address{Mathematics department, Yale university, New Haven, CT 06520}
\email{hee.oh@yale.edu}
\thanks{This work was partially supported by ISF-Moked grant 2095/19 (Landesberg), ERC 2020 grant no.~833423 (Lindenstrauss) and the NSF grant 0003086 (Oh)}
\begin{abstract}
Let $G=\prod_{i=1}^{\mathsf r} G_i$ be a product of simple real algebraic groups of rank one and  $\Gamma$ an Anosov subgroup of $G$  with respect to a minimal parabolic subgroup.
 For
 each $\v$ in the interior of a positive Weyl chamber, let $\mathcal R_\v\subset \Gamma\backslash G$ denote the Borel subset of all points 
with recurrent $\exp (\mathbb R_+ \v)$-orbits.   
For a maximal horospherical subgroup $N$ of $G$, we show that the $N$-action on ${\mathcal R}_\v$ is uniquely ergodic if $\mathsf r=\op{rank}(G)\le 3$ and $\v$  belongs to the interior of the limit cone of $\Gamma$,
and that there exists no $N$-invariant {Radon} measure on $\mathcal R_\v$  otherwise.
\end{abstract}

\maketitle

\tableofcontents
\section{Introduction}
Let $G$ be a connected semisimple real algebraic group, and $\Gamma<G$ be a Zariski dense discrete subgroup.
Let $N$ be a maximal horospherical subgroup of $G$, which is unique up to conjugation.
We are interested in the study of $N$-invariant ergodic Radon measures on the quotient space $\Gamma\ba G$ (from now on, all measures we will consider are implicitly assumed to be Radon measures).
When $\Gamma$ is a uniform lattice in $G$, 
the $N$-action on $\Gamma\ba G$ is known to be uniquely ergodic, that is,
 there exists  a unique  $N$-invariant ergodic measure on $\Gamma\ba G$, up to proportionality,  which is the $G$-invariant measure. This
 result is due to Furstenberg \cite{Fu} for $G=\PSL_2(\br)$ and Veech \cite{Vee} in general.
   Dani \cite{Da} classified all $N$-invariant ergodic measures for a general lattice $\Ga$. Later, Ratner \cite{Ra} gave a complete classification of all invariant ergodic measures for any unipotent subgroup action when $\G$ is a lattice of $G$.

When $G$ is of rank one
and $\Gamma$ is geometrically finite,
there exists a unique $MN$-invariant ergodic measure on $\Ga\ba G$, not supported on a closed $MN$-orbit, where $M$ is a maximal compact subgroup of the normalizer of $N$, called the Burger-Roblin measure. 
 This result is due to Burger \cite{Bu} for convex cocompact subgroups of $\PSL_2(\br)$
with critical exponent bigger than $1/2$, and to Roblin \cite{Ro} in general. 
For $G\not\simeq \op{SL}_2(\br)$, Winter \cite{Win} showed that the Burger-Roblin measure is $N$-ergodic, and hence the $N$-action on $\Gamma\ba G$ 
is essentially uniquely ergodic. This relies on the fact that $M$ is connected. Indeed, for $G\simeq \op{SL}_2(\br)$ where $M=\{\pm e\}$,  the Burger-Roblin measure has one or two $N$-ergodic components depending on $\G$ (cf. \cite[Thm. 7.14]{LO2}).

For geometrically infinite groups, there may be a continuous family of $N$-invariant ergodic measures,
as first discovered by Babillot and Ledrappier (\cite{Bab}, \cite{BL}).  See (\cite{Sa1}, \cite{Sa2}, \cite{Led}, \cite{LS}, \cite{OP}, \cite{LL}, \cite{L}) for partial classification results in the rank one case. 

 In this paper, we obtain a measure classification result for the $N$-action on Anosov homogeneous spaces $\G\ba G$ which surprisingly depends on the rank of $G$: on the recurrent set in an interior direction of the limit cone of $\Gamma$, the $N$-action is uniquely ergodic  if
 $\op{rank} G\le 3$, and admits no invariant measure if
 $\op{rank} G> 3$.

 When the rank of $G$ is one, the class of Anosov subgroups coincides with that of Zariski dense convex cocompact subgroups. To define it in general,
  let $P$ be a minimal parabolic subgroup of $G$. Let $\cal F$ denote the Furstenberg boundary $G/P$, and $\cal F^{(2)}$ the unique open $G$-orbit in $\cal F\times \cal F$.
A Zariski dense discrete subgroup $\Ga<G$ is called an {\it Anosov subgroup} (with respect to $P$) if it is a  finitely generated word hyperbolic group which admits a $\Ga$-equivariant embedding $\zeta$ of the Gromov boundary $\partial \Ga$  into $\cal F$ such that $(\zeta(x),\zeta(y))\in\cal F^{(2)}$ for all $x\ne y$ in $\partial\Ga$. First introduced by Labourie \cite{La} as the images of
 Hitchin representations of surface groups, this definition  is due to Guichard and Wienhard \cite{GW}. The class of Anosov groups in particular includes any Zariski dense Schottky subgroup (cf. \cite{Qu}, \cite[Lem. 7.2]{ELO}).

Let  $P=AMN$ be the Langlands decomposition of $P$, so that $A$ is a maximal real split torus of $G$, $M$
is a compact subgroup which commutes with $A$ and $N$ is the unipotent radical of $P$.
Fix a positive Weyl chamber $\fa^+\subset\fa=\log A$, and 
denote  by $\cal L_\Gamma\subset \fa^+$ the limit cone of $\Gamma$, i.e.,
$\L_\Ga$ is the smallest closed cone of $\fa^+$ which contains the Jordan projection of $\G$ (see \eqref{Jordan} for definition). It is known that if $\Gamma$ is Zariski dense,
$\L_\Ga$ is a convex cone with non-empty interior \cite[Thm. 1.2]{Be}.
We denote by $\La\subset \F$ the limit set of $\Ga$, which is the unique $\Gamma$-minimal closed subset of $\F$. Then
 $$\cal E:=\{[g]\in \Gamma\ba G: gP\in \Lambda\}$$ is the unique $P$-minimal closed subset of $\Gamma\ba G$.
For each vector $\mathsf v\in \inte \fa^+$, define the following  directional recurrent subset of $\E$:
\be\label{eu} \R_{\mathsf v}=\{x\in \Gamma\ba G: x \exp (t_i \v) \text{ is bounded  for some $t_i\to +\infty$}\}.\ee
It is easy to see that $\R_\v=\emptyset$ unless $\v\in \L_\Ga$.
Since  $\v\in \inte \fa^+$ and $AM$ centralizes $\exp (\br \v)$, $\R_{\mathsf v}$ is a $P$-invariant  dense Borel subset of $\cal E$ when it is non-empty. In particular,
 $\R_{\mathsf v}$ is either  co-null or null  for any $N$-invariant ergodic measure on $\G\ba G$.
 We are interested in understanding $N$-invariant ergodic measures supported on~$\R_\v$.
 
 In the rest of the introduction, 
 we assume that 
 $$G=\prod_{i=1}^{\mathsf r} G_i$$ where each $G_i$ is a rank one simple real algebraic group; hence $\r=\op{rank} G$. While $G_i$ can be isomorphic to $\op{PSL}_2(\br)$, we exclude the case when $G_i$ is isomorphic to $\op{SL}_2(\br)$ in order to ensure that $P$ is connected. We let $\G<G$ be an Anosov subgroup. For each $\v\in \inte \L_\Ga$, we denote by $m_{\v}^{\BR}$ the $MN$-invariant Burger-Roblin measure for the direction $\v$ 
(see \eqref{eq.BR}). For Anosov subgroups, it was shown by Lee and Oh that the family $\{\m_\v^{\BR}:\v\in \inte \L_\Ga\}$
gives all $N$-invariant ergodic and $P$ quasi-invariant measures on $\cal E$, up to proportionality (\cite{LO1}, \cite{LO2}).

 The main result of this paper is as follows: 
 \begin{thm} \label{m1} Let $\Ga<G$ be an Anosov subgroup and  $\v\in \inte\fa^+$.
 \begin{enumerate}
\item For $\r\le 3$ and $\v\in \inte \L_\Ga$, the $N$-action on $\R_\v$ is uniquely ergodic.
More precisely, $\m_{\v}^{\BR}$ is the unique $N$-invariant measure supported on
$\R_{\mathsf v}$, up to proportionality.  

\item For $\r> 3$ or $\v\notin \inte\L_\Ga$, there exists no $N$-invariant measure supported on
$\R_\v$.
\end{enumerate}
 \end{thm}
 
This theorem uses the result by
Burger, Landesberg, Lee and Oh \cite{BLLO} that
 $\R_\v$ is a co-null (resp. null) set for $\m_\v^{\BR}$ for $\r\le 3$ (resp. $\r> 3$), which was developed simultaneously, in part for the purpose of this work.

We note that the unique ergodicity as in (1) implies that  $\m_\v^{\BR}$ is $N$-ergodic, reproving some special cases of \cite[Thm. 1.1]{LO2}. When $\r=1$ and $\Gamma$ is a convex cocompact subgroup of $G$,  this theorem recovers the unique ergodicity of the $N$-action on
$\cal E$.

We deduce  the following classification of $N$-ergodic measures supported on the directional recurrent set $$\R:=\cup_{\v\in \inte\fa^+} \R_\v.$$
A measure $\mu$ on $\Ga\ba G$ is said to be supported on $\R$ if the complement of $\R$ is contained in a $\mu$-null set.
\begin{cor}\label{m2} The space $\mathcal M$ of all $N$-invariant ergodic measures supported on $\R$ is given by
\begin{equation*}
    \mathcal M=\begin{cases} \mathcal \{\m_{\v}^{\BR}: \v\in \inte\L_\Ga\}&\text{for $\r\le 3$} \\
    \emptyset &\text{for $\r> 3$}. \end{cases}
\end{equation*}
\end{cor}

We apply our theorem to some concrete examples considered in \cite{Bu2}.
Let $\Sigma$ be a surface subgroup with two convex cocompact realizations in rank one Lie groups $G_1$ and $G_2$. For each $i=1,2$,
denote by $\pi_i:\Sigma\to G_i$ an injective homomorphism with Zariski dense image. We assume that $\pi_2\circ \pi_1^{-1}$ does not extend to
an algebraic group isomorphism $G_1\to G_2$.

It is easy to check that $\Gamma_{\pi_1, \pi_2}:=\{(\pi_1(\gamma), \pi_2(\gamma)):\gamma\in \Sigma\}$ is an Anosov subgroup of  $G:=G_1\times G_2$. 
\begin{cor} For $\Ga=\Ga_{\pi_1, \pi_2}$ as above, the $N$-action on $\R_\v$ is uniquely ergodic  for
 each $\v\in \inte \L_\Gamma$.
\end{cor}

\subsection*{On the proof of Theorem \ref{m1}}
In the rank one case, i.e., when $\Ga$ is convex cocompact, Theorem \ref{m1} follows from the combined works of
Roblin \cite{Ro} and Winter \cite{Win} (see also \cite{MO} and \cite{Sc} for $G=\op{SO}^\circ(n,1)$ case). These proofs are all based on the
finiteness and the strong mixing property of
the Bowen-Margulis-Sullivan measure. In the higher rank case, although there exists an analogous measure (which is also called the Bowen-Margulis-Sullivan measure)  for each direction $\v\in \inte \L_\Ga$, this is an infinite measure \cite[Cor. 4.9]{LO1} and it is not clear how to extend the approaches of the aforementioned papers. We henceforth follow  an approach
of the recent work of Landesberg and Lindenstrauss \cite{LL} for the case $G=\SO^\circ(n,1)$ which is
in the spirit of Ratner's work. 
The main technical result we prove in this paper is the following: 
\begin{prop}\label{pm1}
Let $\Gamma$ be a Zariski dense discrete subgroup of $G$ and $\v\in \inte \fa^+$.
Then any $N$-invariant ergodic  measure $\mu$
on $ \R_\v$ is $P$-quasi-invariant.
\end{prop}
\begin{Rmk} \rm We refer to Theorem \ref{mp} for a more general version, analogous to the main theorem of \cite{LL} for $G=\SO^\circ(n,1)$. \end{Rmk}

%{\ELcomment{ maybe also mention Fraczyk and Gelander's paper ``Infinite Volume and Infinite Injectivity Radius'' in this context?}}

Following \cite{LL},  our proof of Proposition \ref{pm1}
utilizes the geometry observed along the one-dimensional diagonal flow $\exp (\br \v)$ of points in the support of $\mu$ to obtain
an extra quasi-invariance of $\mu$. Roughly speaking,  if, for $\mu$-a.e.~$x \in \Ga\ba G$, we have
$x\exp (t_n \v) g_n =x\exp (t_n \v)$ for some infinite sequence $t_n\to \infty$ and $g_n\in G$
converging to some loxodromic element $g_0\in G$, we show that the generalized Jordan projection of $g_0$ preserves the measure class of $\mu$, provided the attracting fixed point of $g_0$
is in general position with that of $g_0^{-1}$. The last condition always holds in the rank one setting as any two distinct points on $\F$ are in general position. In the higher rank setting,
this property is needed to ensure that the high powers of $g_0$ attract some neighborhood of its attracting fixed point  to itself, which is an underlying key point which makes our analysis possible.

For $G=\SO^\circ(n,1)$,
 the conjugation action of an element of $A$ on $N$ is simply a scalar multiplication, and both the Besicovitch covering lemma and Hochman's ratio ergodic theorem for Euclidean norm balls in the abelian group $N\simeq \br^{\op{dim} N}$ were used in \cite{LL}, in order to control ergodic properties of $N$-orbits.
In our setting where $G$ is a product $\prod G_i$ of rank one Lie groups, 
the horospherical subgroup $N$ is a product $\prod N_i$ of abelian and two-step nilpotent subgroups and the conjugation action by $\exp (t\v)$ scales $N_i$'s by different factors. The existence of $\exp (t\v)$-invariant family of quasi-balls satisfying the Besicovitch covering property in this case is a consequence of the work of Le Donne and Rigot \cite[Thm. 1.2]{DR}. This is precisely the main reason for our assumption that $G$ is the product of rank one Lie groups. We note that in the higher rank case, the ratio ergodic theorem with respect to this family of quasi-balls in our $N=\prod N_i$, is available only when $N$ is abelian \cite{Dooley_Jarrett}.\footnote{We mention that the only case when the ratio ergodic theorem is known and $N$ is not abelian is when $G \simeq \op{SU}(n,1)$ and $N$ is Heisenberg \cite{Jarrett-Heisenberg}.}
To sidestep the lack of the ratio ergodic theorem in the generality we need,
we use in this paper a modified argument relying only on the Besicovitch covering property.
 In addition to technical difficulties arising in the higher rank setting and from the fact that $N$ is not necessarily abelian, our proof of Proposition \ref{pm1} is different from \cite{LL} also in this aspect.
 
 Theorem \ref{m1} is then deduced from Proposition \ref{pm1} together with the classification of $\Ga$-conformal measures on $\La$ of \cite{LO1} (Theorem \ref{class}) and  the dichotomy on the recurrence property of 
 the Burger-Roblin measures according to the rank of $G$, obtained in \cite{BLLO} (Theorem \ref{bb2}).

\medskip

\subsection*{Rank one groups} 
 While the main emphasis in this paper is on the higher rank case, one can also deduce the following new result for all rank one groups.
Given Theorem \ref{mp} and the description of $N$-ergodic invariant and $P^\circ$-quasi invaiant measures (cf. \cite[Lem. 5.2]{LL}, \cite[Prop. 7.2]{LO2}), the following corollary can be proved almost verbatim  as \cite[Cor. 1.1, 1.2]{LL} and \cite[Thm. 1.5]{L} where similar statements were established for $G=\SO^\circ(n,1)$.

For $y\in \Ga\ba G$, we denote by $\injrad (y) $ the supremal injectivity radius at $y$. 
%We denote by $\delta_\Ga$ the critical exponent of $\Ga$. 
\begin{cor}\label{cor_injrad_intro}
	Let $\Gamma$ be a Zariski dense discrete subgroup of a simple real algebraic group $G$ of rank one. 
	Let 
	$\mu$ be an $ N $-invariant ergodic measure supported on $ \cal E$.
\begin{enumerate}[leftmargin=*]
    \item 
	If the injectivity radius on $\Gamma\ba G$ is uniformly bounded \emph{away from $0$}, then at least one of the following holds:
	\begin{enumerate}
		\item $ \mu $ is quasi-invariant under some loxodromic element of $ P $,
		\item $ \lim_{t \to \infty} \injrad( x \exp t\v ) = \infty $ for $ \mu $-a.e.~$ x $ and $\v\in \inte\fa^+$. 
	\end{enumerate}
	
\item If the injectivity radius on $\Gamma\ba G$ is uniformly bounded \emph{from above} or if $\Gamma$ is a normal subgroup
	of a geometrically finite subgroup of $ G $, then either:
\begin{enumerate}
	\item $\mu$ is proportional to $\m^{\BR}_\nu|_Y$ for some
	$\G$-conformal measure $\nu$ on $\La$ and a $P^\circ$-minimal subset $Y\subset \Ga\ba G$ (see \eqref{eq.BR} for the definition of $\m_\nu^{\BR}$), or
	\item $ \mu $ is supported on a closed $MN$-orbit.
\end{enumerate}
\end{enumerate}
\end{cor}

\noindent We remark that by a recent work of Fraczyk and Gelander \cite{FG}, the injectivity radius on $\Ga\ba G$ is never bounded from above when $G$ is simple with
$\op{rank}G\ge 2$ and $\text{Vol}(\Ga\ba G)=\infty$.

\begin{Rmk} For $\Gamma$ geometrically finite, an atom of a $\Ga$-conformal density
is necessarily a parabolic limit point which yields a closed $MN$-orbit, and
the so-called Patterson-Sullivan measure, say, $\nu_0$, is the unique atom-free $\Ga$-conformal measure on $\La$ \cite{Su}.
Therefore Corollary
\ref{cor_injrad_intro}(2) implies the essential unique ergodicity for the $N$-action as well as the $N$-ergodicity of
$\m_{\nu_0}^{\BR}|_Y$ for each $P^\circ$-minimal subset $Y$.
Noting that the proofs given in \cite{MO} and \cite{Sc} on the $N$-unique ergodicity for $\op{SO}^\circ (n,1)$ rely on the ratio ergodic theorem for the abelian subgroup $N$ which is not available for a general rank one group,
our paper gives the only alternative proof for a general rank one case after Roblin and Winter (\cite{Ro}, \cite{Win}).
\end{Rmk}

\medskip
\subsection*{Organization}
In section 2, we set up notations and recall basic definitions. In section 3, we deduce the Besicovitch covering lemma for our setting from \cite{DR} and state several consequences including the maximal ratio inequality. In section 4, we prove Theorem \ref{mp}, which is the main technical result of this paper. In section 5,
we prove Theorem \ref{normal} which in particular implies Proposition \ref{pm1}, using Theorem \ref{mp} together with some properties of Zariski dense subgroups. In section 6, we specialize to Anosov subgroups and prove Theorem \ref{m1}.

\medskip 
We close the introduction with the following open problems.
\begin{Q}\rm For $\mathsf r\le 3$ and $\G$ Anosov,
is any $N$-invariant ergodic measure
on $\E$  necessarily supported on $\R_\v$ for some $\v\in \inte \L_\Ga$?
\end{Q}
%\begin{Q}\rm Let $\G$ be Anosov.
%\begin{enumerate} 
%\item For $x\in \E$,  is it possible that
%$\liminf_{t\to +\infty} \injrad x\exp (t\v)=\infty$ for all $\v\in \inte \fa^+$?
%\item For $\mathsf r\le 3$, is any $N$-invariant ergodic measure
%on $\E$  necessarily supported on $\R_\v$ for some $\v\in \inte \L_\Ga$?
%\end{enumerate}
%\end{Q}

\medskip
\noindent{\bf Acknowledgement.} We would like to thank Emmanuel Breuillard
and Amir Mohammadi for useful conversations on this work.
\section{Preliminaries}
Let $G$ be a connected, semisimple real algebraic group.   We fix, once and for all, a Cartan involution $\theta$ of the Lie algebra $\mathfrak{g}$ of $G$, and decompose $\fg$ as $\frak g=\frak k\oplus\mathfrak{p}$, where $\fk$ and $\fp$ are the $+ 1$ and $-1$ eigenspaces of $\theta$, respectively. 
We denote by $K$ the maximal compact subgroup of $G$ with Lie algebra $\fk$.
Choose a maximal abelian subalgebra $\fa$ of $\frak p$. Choosing a closed positive Weyl chamber $\fa^+$ of $\fa$, let $A:=\exp \frak a$ and $A^+=\exp \frak a^+$. 
The centralizer of $A$ in $K$ is denoted by $M$, and we set 
$N^-$ and $N^+$ to be the contracting and expanding horospherical subgroup: for  $a\in \inte A^+$,
  $$N^{\pm} =\{g\in G: a^{-n} g a^n\to e\text{ as $n\to \mp \infty$}\}.$$
 We set $P^{\pm}=MAN^{\pm}$, which are minimal parabolic subgroups.
As we will be looking at the $N^-$-action in this paper, we set  $N:=N^-$ and $P=P^-$ for notational simplicity. We also set $L=MA=P\cap P^+$.

  Let $w_0\in N_K(A)$ be the Weyl element satisfying
  $\op{Ad}_{w_0}\fa^+= -\fa^+$. Then $w_0$ satisfies
$w_0P^-w_0^{-1}=P^+$.
  For each $g\in G$, we define 
   $$g^+:=gP\in G/P\quad\text{and}\quad g^-:=gw_0P\in G/P.$$

 Let  $\F=G/P$ and $\F^{(2)}$ denote the unique open $G$-orbit in $\F\times \F$:
$$\F^{(2)}=G(e^+, e^-)=\{(g^+, g^-)\in \cal F\times \cal F: g\in G\}.$$ 
We say that $\xi, \eta$ are in general position if $(\xi, \eta)\in \F^{(2)}$.

Any element $g\in G$ can be written as the commuting product $g_hg_e g_u$, where $g_h$, $g_e$ and $g_u$
 are unique elements which are conjugate to elements of $A^+$, $K$ and $N$, respectively. 
We say $g$ is  {\it loxodromic} if $g_h\in \varphi (\inte A^+)\varphi^{-1}$ for some $\varphi\in G$, and write 
\be\label{Jordan} \la^A(g):= \varphi^{-1}g_h \varphi\in \inte A^+\ee calling it the Jordan projection of $g$.  We set 
  \be\label{yg} y_g:=\varphi^+;\ee
 this is well-defined independent of the choice of $\varphi$.
We note that $g$ fixes $y_g$ and
for any $h\in N^+$, $\lim_{k\to \infty}
g^k (\varphi h e^+) =y_g$, uniformly on compact subsets of $N^+$, and
for this reason, $y_g$ is called the attracting fixed point of $g$.
  
\subsection*{Bruhat coordinates} The product map $N\times A\times M\times N^+\to G$ is injective and its image is Zariski open in $G$.
For $g\in G$ and $n\in N$ with $gn\in NAMN^+$,
we write
\begin{equation}\label{eq.Bruhat}
gn=b^{N}(g,n)b^{AM}(g,n)b^{N^+}(g,n)
\end{equation}
where $b^N(g,n)\in N,\,
 b^{AM}(g,n)\in AM, \, b^{N^+}(g,n)\in N^+$ are uniquely determined.
For each subgroup $\star=N, AM $ or $ N^+$, $b^\star (g, n) $ is a smooth function for each $g\in G$ and $n\in N$ whenever it is defined.

For convenience,
for $\xi=ne^-$ with $n\in N$ and $g \in G$ with $g\xi\in Ne^-$,
we set $$b^\star(g,\xi):=b^\star(g,n).$$
If $g\in G$ is a loxodromic element with $ y_g\in Ne^-$,
the following
generalized Jordan projection of $g$ is well-defined:
$$
\la(g)=b^{AM}(g, y_g).
$$
We mention that the condition $y_g\in Ne^-$ implies that there exists $\varphi\in N N^+ $ such that
$g=\varphi a^{-1} m \varphi^{-1}$ for unique $a\in \inte A^+ $ and $m\in M$. In this case, $\la(g)=a^{-1}m$.
In particular, the $A$-component of $\la(g)$ coincides with $\la^A(g^{-1})$. 
If $g$ is not loxodromic, we set $\la(g)=e$.

\section{Covering lemma for $\exp t\v$-conjugation invariant balls}\label{hyp}
In the rest of the paper, let  $G:=\prod_{i=1}^{\r} G_i$ where $G_i$ is a connected simple real algebraic group of rank one.
For each $1\le i\le \mathsf r$, we identify $G_i$ with the subgroup
$\{(g_j)_j\in \prod_j G_j: g_j=e\; \text{ for all $j\ne i$}\}<G$
and we set $H_i:=H\cap G_i$ for any subset $H\subset G$.
We have $A=\prod_i A_i$ and $A^+=\prod_i A_i^+$ where $A_i$ is a one-parameter diagonalizable subgroup of $G_i$. Let $\alpha_i$ denote the simple root
of $G_i$ with respect to $A_i$.
  The  subgroup $N=N^-$ is of the form $N=\prod_i N_i$, where $N_i$ is the contracting
horospherical subgroup of $G$ for $A_i^+$ and $P=\prod P_i$ for $P_i=M_iA_iN_i$. We set $\cal F_i=G_i/P_i$.

As $G_i$ has rank one, $N_i$ is a connected simply connected nilpotent subgroup of at most $2$-step. 
Let $\frak n_i$ denote the Lie algebra of $N_i$. 
When $\frak n_i$ is abelian, for each $a_i\in A_i$, $\op{Ad}_{a_i}|_{\frak n_i}$ is the multiplication by $ e^{\alpha_i(\log a_i)} $.
When $\frak n_i$ is a $2$-step nilpotent, we can write $\frak n_i=\frak n_{i_1}\oplus \frak n_{i_2}$ where $[\frak n_{i_1}, \frak n_{i_1}]\subset \frak n_{i_2}$ and $\frak n_{i_2}$ is the center of $\frak n_i$. We have that  for $a_i\in A_i$, $\op{Ad}_{a_i}|_{\frak n_{i_1}}  =e^{\alpha_i(\log a_i)}$ and 
$\op{Ad}_{a_i}|_{\frak n_{i_2}}  =e^{2\alpha_i(\log a_i)}$ (cf. \cite{Mos}).

We call a function $d:N\times N\to [0, \infty)$ a quasi-distance on $N$ if it is symmetric, $d(x,y)=0$ iff $x=y$,
and there exists $C=C(d)\ge 1$ such that 
\be\label{CD} d(x,y)\le C (d(x,z)+ d(z, y))\quad\text{ for all $x,y,z\in N$.} \ee

For $s>0$ and $x\in N$, we set $B_d(x, s)=\{y\in N: d(x, y)< s\}$. 
For simplicity, we write $B_d(s):=B_d(e,s)$. Note that whenever $d$ is left-invariant, $B_d(x,s)= x B_d( s)$ for all $x\in N$ and $s>0$. 

When $N$ is abelian, it is well-known that Euclidean norm-balls of $N$ satisfy the Besicovitch covering property. In general,
we deduce the following from~\cite{DR}.
\begin{prop} \label{REH}For any $\mathsf v \in \inte \fa^+\cup\{0\}$,
there exists a continuous left-invariant quasi-distance $d=d_{\v}$ on $N$ such that the family of balls $\{B_d(u,s)=uB_d(s): u\in N, s>0\}$ satisfies the Besicovitch covering property.
That is,
there exists a constant ${\kappa_{\v}}>0$, depending only on $ d_{\mathsf v}$, such that for any bounded subset $S\subset N$, and any cover $\{ uB_d(t_u): u\in S\}$ of $S$, for some positive function $u\mapsto t_u$ on $S$, there exists a countable subset $F\subset S$
such that $\{ uB_d(t_u): u\in F\}$ covers $S$ and
$$\sum_{u\in F}  \mathbbm{1}_{u B_d(t_u)}\le
{\kappa_{\v}}.$$
Moreover, if $\v=0$, we can take $d_{\v}=d_0$ to be a distance, and
 if $\v\ne 0$, we have
\be\label{cong} B_d(e^t r)= \exp (t\mathsf v) B_d(r) {\exp (-t\mathsf v)}\quad\text{ for all $t\in \br$ and $ r>0$}.\ee
\end{prop}
\begin{proof} 
For $\lambda\ge 1$, consider
the Lie algebra homomorphism $\frak n\to\frak n$ given by $\delta_{\lambda}X =\op{Ad}_{\exp ((\log \lambda) \v)} X$.
Let $I:=\{i:\frak n_i\text{ abelian} \}$ and $J:=\{i:
\frak n_i \text{ is of $2$-step}\}$. Set $t_i:={\alpha_i (\v)}\ge 0$.
For $i\in I$, set $V_{t_i}:=\frak n_i$ and for $i\in J$, set $V_{t_i}:=\frak n_{i_1}$ and $V_{2t_i}:=\frak n_{i_2}$.
Since $\delta_\lambda$ acts on each $V_{t_i}$ (resp. $V_{2t_i})$ by $\lambda^{t_i}$ (resp. $\lambda^{2t_i}$),
and $\sum_{i\in I}V_{t_i}+\sum_{i\in J} V_{2t_i}$ is the center of $\frak n$, it follows that $\frak n= \left(\oplus_{i\in I\cup J} V_{t_i}\right) \oplus \left(\oplus_{i\in J} V_{2t_i}\right)$ provides commuting different layers for the family
$\{\delta_\lambda| \lambda>0\}$ in the terminology of \cite{DR}. Hence \cite[Thm. 1.2]{DR} provides the required quasi-distance
such that $d(\delta_\lambda (n_1), \delta_\lambda(n_2))= \lambda d(n_1, n_2)$ where $\delta_\lambda(n)= e^{(\log \lambda)\v} n e^{-(\log \lambda) \v}$ also denotes the Lie group isomorphism of $N$ induced from $\delta_\la$. For $\lambda =e^t$, this implies \eqref{cong}.
 If $\v=0$, then $t_i=2t_i=0$ for all $i$, and hence $\frak n=V_0$. Now \cite[Cor. 1.3, Def. 2.21]{DR} implies that $d_0$ can be taken to be a distance.
\end{proof}

Indeed, an explicit construction of $d_\v$ has been given in \cite{DR}: for $\v\in \inte \fa^+$,
for $(X_i)_i, (Y_i)_i\in \prod_i N_i$, and
\be \label{fixed} d_\v ((X_i)_i,  (Y_i)_i) =\max_i \mathsf d_i (X_i,  Y_i)^{1/\alpha_i(\v)}\ee
where $\mathsf d_{i}$ is a left invariant metric on $N_i$
induced from an Euclidean norm on $\frak n_i$.

For each $\v\in \inte\fa^+$ (resp. $\v=0$),
we fix a quasi-distance $d_\v$ as above (resp. a distance $d_0$),  and write 
 for any $\e>0$ and $u\in N$,
 \be\label{bv} B_\v(u, \e):= B_{d_\v}(u, \e),\quad\text{and}\quad 
B_\v(\e):= B_{d_\v}(\e).\ee

We denote by $m$ a Haar measure on $N$ and by $2\rho$ the sum of all positive roots, i.e.,
 $2\rho=\sum_{i=1}^{\mathsf r} \alpha_i (\text{dim} N +\text{dim} Z(N) )$, where $Z(N)$ denotes the center of $N$.  
For $\v\ne 0$, we have from \eqref{cong} that for any $R>0$ and $u\in N$, \be\label{vol} m(B_\v(u,R)) = R^{2 \rho (\v)}m(B_\v(u,1)). \ee
 
 For $\v=0$, $d_0$ is a left-invariant metric and by \cite{Gu} (see also \cite{Br}), we have
 \be\label{vol2} m(B_0(u,R)) =O(R^{\text{dim} N +\text{dim} Z(N)}). \ee

\begin{lemma} \label{ka1} Fix $\v\in \inte \fa^+$, $\beta>0$, $0<\eta_1<\eta_2$ and let $u\mapsto t_u$ be a positive function on $N$.
Consider the two collections of balls $\{B_\v(u, e^{t_u} {\eta_i})  : u\in N, t_u>0\}$ for $i=1,2$.
Then for any bounded subset $S\subset N$, there exists a countable subset $F\subset S$ such
that $\{ B_\v(u_i, e^{t_{u_i}} {\eta_1}) :u_i\in F\}$ covers $S$ and 
the following holds:
for each $u_j\in F$, 
\begin{equation*} 
\#\{u_i\in F: B_\v(u_i, e^{t_{u_i}} {\eta_1}) \subset B_\v(u_j, e^{t_{u_j}} {\eta_2}), |t_{u_i} -t_{u_j} |\le \beta\}\le  \kappa_* (\v,\beta, \eta_1,  \eta_2) \end{equation*}
where $\kappa_* (\v,\beta, \eta_1,  \eta_2):= \frac{m(B_\v({\eta_2}))}{m(B_\v({\eta_1})) } e^{\|2\rho\|\beta} {\kappa_{\v}} .$
\end{lemma}

\begin{proof} Set $B_{u}:=B_\v(u, e^{t_u} \eta_1)$
and $ C_{u}:=B_\v(u, e^{t_u} \eta_2)$.
Let $F\subset S$ and $\{B_{u_i}:u_i\in F\}$ be respectively the countable subset and the corresponding countable subcover of $S$ given by Proposition \ref{REH}.
Fix $u_j\in F$. Suppose that $B_{u_1}\cup \cdots \cup  B_{u_p}\subset C_{u_j}$ and that $ |t_{u_i} -t_{u_j} |\le \beta$ for all $1\le i \le p$.
Since $$\sum_{i=1}^p \mathbbm{1}_{B_{u_i}}\le \kappa_{\v}\cdot  \mathbbm{1}_{\cup_{i=1}^p B_{u_i}} ,$$
we have 
\be\label{cuj} m(C_{u_j}) \ge m(\cup_{i=1}^p B_{u_i}) \ge \frac{1}{{\kappa_{\v}}} \sum_{i=1}^p m(B_{u_i}).\ee

Using \eqref{vol}, we get
$$
m(B_{u_i})\ge e^{-\|2\rho\|\beta} m(B_{u_j}) ,\text{ and }
m(C_{u_j})=\frac{m(B_{u_j}) m(B_\v({\eta_2}))}{m(B_\v({\eta_1}))} .$$
It then follows from \eqref{cuj}:
$$\frac{m(B_\v({\eta_2}))}{ m(B_\v({\eta_1}))} \ge \frac{p}{{\kappa_{\v}}}  e^{-\|2\rho\|\beta}, \text{ and hence }
p\le \frac{m(B_\v({\eta_2}))}{m(B_\v({\eta_1})) }{\kappa_{\v}}  e^{\|2\rho\|\beta},$$
proving the claim.
\end{proof}

The following is a consequence of the polynomial growth of the quasi-balls
$B_\v (t)$ in $N$: \begin{lem} \label{lem.x1} Let $\mu$ be  an $N$-invariant ergodic measure  on a Borel space $Z$ and fix
 $\v\in \inte\fa^+\cup\{0\}$.
For any bounded Borel subset $\Omega $ of $Z$ with $\mu(\Omega)>0$,
there exists a co-null subset $Z'$ (depending on $\Om$) such that
for all $x\in Z'$, we have the following: for any $r, \e>0$,  there exists a sequence $t_i\to \infty$ such that
\be\label{coo} \frac{\int_{ B_\v({t_i+r} )}\mathbbm{1}_{\Om}(xn)\,dn}{\int_{B_\v({t_i})}\mathbbm{1}_{\Om}(xn)\,dn}\le 1+\epsilon.
\ee
\end{lem}

\begin{proof} For $x\in Z$ and a subset $\Om\subset Z$, we write
\be\label{to} \T_{\Om}(x)=\{u\in N: xu\in \Om\}.\ee 

By ergodicity of $\mu$, we know that $\mu$-almost every $N$-orbit intersects $\Omega$ non-trivially.
Indeed, consider the set
\[ E:=\{ x \in  Z\;|\; m(\T_{\Omega} (x)\cap B_{\v}({s_x})) > 0 \text{ for some $s_x>0$}\}.\]
     
If $ x \in E $, then, for any $ u \in N $, there exists $ s > s_x $ satisfying
      \[ B_{\v}({s_x}) \subset u B_{\v}(s) \]
      and consequently
      \[ m(\T_{\Omega}(xu) \cap B_{\v}(s)) = m(\T_{\Omega}(x) \cap uB_{\v}(s)) \geq m(\T_{\Omega} (x) \cap B_{\v}({s_x})) > 0, \]
      implying $ xu \in E $. 
          Hence the set $ E $ is $ N $-invariant. Now, by ergodicity of $ \mu $, the set $ E $ is either null or conull. 
On the other hand, since
      \[ \int_{Z} m(\T_{\Omega}(x) \cap B_{\v}(1)) d\mu(x) = \int_{B_{\v}(1)} \int_{Z} \mathbbm{1}_\Omega(xn) d\mu(x) dn = m(B_{\v}(1))\mu(\Omega)>0, \]
the set $\{x\in Z: m(\T_{\Omega}(x) \cap B_{\v}(1))>0\}$ has positive measure. Therefore $\mu(E)>0 $, and hence $E$ is conull. Set $Z'=E$.
Let $x\in {Z'}$ and $s_x>0$ be such that $m(\T_{\Omega} (x)\cap B_{\v}({s_x})) > 0$. Suppose that $\eqref{coo}$ does not hold for $x$. Then
                   there exists $t_x>s_x$ such that
    for all $t\ge t_x$,
$$m(B_{\v}({t+r}))\ge m(\T_\Om (x)\cap B_{\v}({t+r}))\ge (1+\e) m(\T_{\Om}(x)\cap B_{\v}(t)).$$
It follows that for all $k\ge 1$,
$$m( B_{\v}({t_x+kr})) \ge (1+\e)^k m(\T_{\Om}(x)\cap B_{\v}({t_x})).$$
Since  $m( B_{\v}({t_x+kr}))$ grows polynomially in $k$ by \eqref{vol} and \eqref{vol2},
and since $m(\T_{\Om}(x)\cap B_{\v}({t_x}))>0$, this yields a contradiction.
\end{proof}

A standard consequence of the Besicovitch covering property is the maximal ratio inequality. These are in fact equivalent when considering symmetric averaging sets, see \cite{Hochman} and references therein. For completeness we include below a proof of this implication applicable to our setup:

\begin{lem}[Maximal ratio inequality]\label{maxi}
Let $\mu$ be  an $N$-invariant ergodic measure  on a Borel space $Z$.
 Fix $\v\in \inte\fa^+\cup\{0\}$ and $\alpha >0$.
 For any  bounded measurable subsets $\Omega_1 $ and $\Omega_2$ of $Z$ with $\mu(\Om_2)<\infty$, we have $$\mu(\Om_2\cap E^\dagger) \le 2{\kappa_{\v}  }\alpha^{-1}  \mu(\Om_{1} )$$
where
$$E^\dagger :=
\left\{x\in Z:  \exists R >0 \text{ s.t } {\int_{ B_{\v}(R)}\mathbbm{1}_{\Om_1}(xn)\,dn}\ge \alpha {\int_{B_{\v}(R)}\mathbbm{1}_{\Om_2}(xn)\,dn}\right\}.$$
\end{lem}
\begin{proof} For $R_1\ge 0$, set
$$E(R_1):=
\left\{x\in Z: \exists 0\le R\le R_1 \text{ s.t }  {\int_{B_{\v}(R)}\mathbbm{1}_{\Om_1}(xn)\,dn}\ge \alpha{\int_{B_{\v}(R)}\mathbbm{1}_{\Om_2}(xn)\,dn} \right\}.$$
Since $E(R_1)$ is an increasing sequence of subsets whose union is $E^\dagger$ and $\mu(\Om_2)<\infty$,
it suffices to show that for any $R_1\ge 0$,
$$\mu(\Om_2\cap E({R_1})) \le 2{\kappa_{\v}}\alpha^{-1} \mu(\Om_1) .$$
Fix a compact subset $D=D(R_1)\subset N$ so that $0<m(DB_{\v}({R_1}))\le 2 m(D)$, which is possible in view of \eqref{vol} and \eqref{vol2}.  
Let $\T_\Om(x)$ be defined as in \eqref{to}.
For each $x\in Z$ with $xu\in {E(R_1)}$,
there exists $0\le R_u\le R_1$ such that
$$m(\T_{\Om_2}(x)\cap B_{\v}(u, R_u))\le \alpha^{-1} m(\T_{\Om_1}(x) \cap B_{\v}(u, R_u)).$$

Consider the cover $\mathcal C (x)=\{B_{\v}(u, R_u) :  u\in D\cap \T_{ E(R_1)}(x)\}$
of the subset $ D\cap \T_{ E(R_1)}(x)$.
By Proposition \ref{REH}, we can find a countable subset $I_x\subset N$
such that the family $\{B_{\v}(u, R_u): u\in I_x\} \subset \cal C(x)$
covers  $D\cap \T_{ E(R_1)}(x)$ and  $$\sum_{u\in I_x}  {\mathbbm 1}_{B_{\v}(u, R_u)}\le\kappa_{\v}\mathbbm 1_{DB_\v(R_1)}.$$

We obtain:
\begin{align*}
& \mu(\Omega_2\cap E(R_1) )=\frac{1}{m(D)} \int_{Z} \int_{D}\mathbbm{1}_{\Om_2\cap E(R_1)}(xn)\,dn d\mu(x)
\\&= \frac{1}{m(D)} \int_{Z} {m(D\cap \T_{\Om_2\cap E(R_1)}(x) \cap(\cup_{u\in I_x}B_{\v}(u, R_u) ))}d\mu(x) \\
&\leq \frac{1}{m(D)} \int_{Z} \sum_{u\in I_x}  {m(\T_{\Om_2}(x)\cap B_{\v}(u, R_u) )}d\mu(x)\\
&\leq\frac{1}{\alpha \cdot m(D)}\int_{Z} \sum_{u\in I_x} {\int_{N} \mathbbm{1}_{B_{\v}(u, R_u)}(n) \cdot \mathbbm{1}_{\Om_1}(xn)}dn\, d\mu(x) \\
&= \frac{1}{\alpha\cdot  m(D)}\int_{N}\int_{Z} \left(\sum_{u\in I_x}\mathbbm{1}_{B_{\v}(u, R_u)}(n)\right)  \mathbbm{1}_{ {\Om_1}}(xn)d\mu(x) \,dn
\\&\leq \frac{ \kappa_{\v} }{\alpha \cdot  m(D)}\int_{DB_{\v}(R_1)}\int_{Z}\mathbbm{1}_{\Om_1}(xn)d\mu(x) \,dn
\\&= \frac{\kappa_{\v}\cdot m(DB_{\v}( R_1) ) }{\alpha\cdot m(D)} \mu (\Om_1)\\
&\le 2 \frac{\kappa_{\v} }{\alpha } \mu (\Om_1).
\end{align*}
\end{proof}

\section{Scenery along $\exp (\br_+\v)$-flow and quasi-invariance}
As before,  let $G:=\prod_{i=1}^{\r} G_i$ where $G_i$ is a connected simple real algebraic group of rank one. 
Let $\Ga$ be a discrete subgroup of $G$.
Let $\mu$ be an $N$-invariant ergodic measure on $\Ga\ba G$. In the whole section, we fix a vector $\mathsf v\in \inte\fa^+$, and  set $$a_t:=\exp (t\v)\quad\text{ for $t\in \br$.}$$

For all $x\in \Ga\ba G$, define
$$
\cal S_x(\v):=\limsup_{t\to +\infty} a_t^{-1}g^{-1}\Ga g a_t = \limsup_{t\to +\infty} \op{Stab}_G(xa_t).$$
The $\limsup_{t \to +\infty}$ above is the topological limit superior, i.e., the collection of all accumulation points; hence we may otherwise write
\[ \cal S_x(\v) = \bigcap_{n=1}^\infty \overline{\bigcup_{t>n} a_{-t}g^{-1}\Gamma g a_t}. \]
As $\v\in \inte \fa^+$, we have $\cal S_{xn}(\v)=\cal S_{x}(\v)$ for all $n\in N$, and hence the measurable map $x\mapsto \cal S_x$ is $N$-invariant. Since $\mu$ is $N$-ergodic, there exists a closed subset $\cal S_\mu(\v) $ of $G$ for which $\cal S_x(\v)=\cal S_\mu(\v)$ for $ \mu $-a.e.~$ x \in \Gamma \ba G $.

 For $\xi, \eta\in \cal F$,
we set $$\frak O_{(\xi, \eta)}:=\{h\in G:\text{loxodromic}, (y_{h}, \xi), (y_{h^{-1}}, \eta)\in \F^{(2)}\} .$$

We remark that as $G_i$'s are rank one groups,
for a loxodromic element $h=(h_1, \cdots, h_\r)\in G$ with $h_i\in G_i$
and $\xi=(\xi_1, \cdots, \xi_\r)\in \cal F$ with $\xi_i\in \cal F_i$,
we have 
$(y_h, \xi)\in \F^{(2)}$ if and only if
$y_{h_i}\ne \xi_i$ for all $1\le i\le \r$.

The main  result of this section is the following:
\begin{thm}\label{mp}\label{prop.sch}
We have
\be\label{smu2} \la \left(\cal S_\mu(\v) \cap (\frak O_{(e^+, e^-)}\cup \frak O_{(e^-, e^+) } )\right)\subset \op{Stab}_{G}([\mu])\ee
  where
$\op{Stab}_G([\mu])$ denotes the stabilizer in $G$ of the measure class of $\mu$.
\end{thm}

When $G$ is of rank one, any loxodromic element of $ G$  belongs to 
either $\frak O_{(e^+, e^-)}$ or
$\frak O_{(e^-, e^+)}$. Therefore \eqref{smu2} is same as saying
$$\la (\cal S_\mu(\v)) \subset \op{Stab}_{G}([\mu]);$$
this generalizes \cite[Thm. 1.3]{LL} to all rank one Lie groups.

Since $\cal S_\mu(\v)^{-1}=\cal S_\mu(\v)$, $\frak O_{{(e^+, e^-)}}^{-1}=
\frak O_{{(e^-, e^+)}}$, and $\op{Stab}_G([\mu])$
is a subgroup of $G$, \eqref{smu2} follows if we show:  
\be\label{smu} \la (\cal S_\mu(\v) \cap \frak O_{{(e^+, e^-)}} )\subset \op{Stab}_{G}([\mu])\ee

\medskip

The rest of this section is devoted to the proof of \eqref{smu}.
We fix the left-invariant quasi-distance $d_\v$ as in \eqref{fixed} and
 set 
$$N_\eta:=B_\v(\eta)\quad\text{for each $\eta>0$}$$
where $B_\v(\eta)$ is defined as in \eqref{bv}.
We set $$t_i:=\alpha_i(\v)>0\quad
\text{for each $1\le i\le \r$.}$$
Since $d_\v=\max_i \mathsf d_i^{1/t_i}$ where $\mathsf d_i$ is a left-invariant metric on $N_i$, for any $\eta>0$, the quasi-ball $N_\eta$ is a product of  balls in $N_i$:  
\be\label{neta} N_\eta=\prod_{{i}=1}^{\r} N_i(\eta^{t_i})\ee where 
$N_i(\eta^{t_i}):=\{ x\in N_i: \mathsf d_i(e_i, x)< \eta^{t_i}\}$ and $e_i$ denotes the identity element of $G_i$.\footnote{ We stress that the notation $N_i$ with subscript $i$ is used solely for the subgroup $G_i\cap N$, whereas $N_\eta, N_\epsilon, $ etc are used for quasi-balls in $N$.}

Fix any loxodromic element $$h_0\in \cal S_\mu(\v) \cap \frak O_{(e^+, e^-)}.$$ Our goal is to  show that
$\la (h_0)\in \op{Stab}_G([\mu])$.

Writing $h_0=(h_1, \cdots, h_{\r})$ component-wise, each $h_i$ is a loxodromic element of $G_i$.
We write $h_i=\varphi_ia_i^{-1}m_i\varphi_i^{-1}$ for some $a_i\in A_i^+-\{e\}$, $m_i\in M_i$ and $\varphi_i\in G_i$ so that $\varphi_i^-=\varphi_ie_i^-\in {\cal F}_i$ and $\varphi_i^+=\varphi_ie_i^+\in {\cal F}_i$ are the unique attracting fixed points of $h_i$ and $h_i^{-1}$ respectively; here $e_i^{\pm}\in \cal F_i$ means the $i$-th component of $e^{\pm}\in \cal F=\prod_i \cal F_i$.
As $G_i$ is of rank one, we have ${\cal F}_i= N_ie_i^-\cup\{e_i^+\}$. Since $h_0\in \mathfrak O_{(e^{+},e^-)}$, we have, for all $i$,
$$\varphi_i^-\ne e_i^+\text{ and } \varphi_i^+\ne e_i^-. $$
We denote by $n_i$ the unique element of $N_i$ such that
\be\label{ni} \varphi_i^-=n_ie_i^-\in N_ie_i^- .\ee

Using the diffeomorphism between $N_i$ and $N_ie_i^-$ given by $n\mapsto ne_i^-$, we may regard
$\d_i$ as a left-invariant metric on $N_i e_i^-$, so that 
\be\label{metric} \d_i (ne_i^-, n' e_i^-)=\d_i(n, n')\quad\text{ for all $n , n'\in N_i$.}\ee

\subsection*{Definition of $\eta_0$.} 
%Fix $0<\eta_0<1$ so that  for all $1\le i \le \r$,
%$$\d_i (\varphi_i^+, e_i^-)>2 \eta_0^{t_i}.
%$$

Since $e_i^-\neq\varphi_i^+$ and hence $e_i^-\in\varphi_iN_ie_i^-$
, there exist $\eta_0>0$ and $J>0$ such that
\begin{equation}\label{eq.J}
N_{\eta_0}e^-\subset \prod_{i=1}^{\mathsf r}\varphi_i N_i(J)e_i^-.
\end{equation}

\begin{lem} There exists $p_0=p_0(h_0)\in \mathbb N$ such that for all $p\ge p_0$, and $1\le i\le \r$, we have
\be\label{contracting1} \d_i(h_i^p z_i,h_i^p z_i')\le \frac{1}{2^{(t_i+1)}} \cdot  \d_i(z_i,z_i')\ee  
for all $z_i$, $z_i'\in \varphi_iN_i(J)e_i^-$.
\end{lem}

\begin{proof} 
Since $(a_i^{-1}m_i)^p ne_i^-= (a_i^{-p}(m_i^p nm_i^{-p}) a_i^p) e_i^-$ and $M_i$ is a compact subgroup normalizing $N_i$,
we have $(a_i^{-1}m_i)^p ne_i^-\to e_i^-$ as $p\to\infty$, uniformly for all $n\in N_i(J
)$. Therefore 
$\varphi_i (a_i^{-1}m_i)^p N_i({J})e_i^-$ is contained in a compact subset of $N_i\varphi_i^-=N_ie_i^-$ for all sufficiently large $p$.
Since $N_ie_i^-$ is endowed with a metric $\d_i$, induced from a Euclidean norm on $\frak n_i$,  the Lipschitz constant $\op{Lip}(\varphi_i|_{(a_i^{-1}m_i)^p N_i(J) e_i^-})$ is well defined and finite.
Since $h_i^p= \varphi_i (a_i^{-1}m_i)^p\varphi_i^{-1}$, we have
\begin{align*}\label{eq.Lip}
&\op{Lip}(h_i^p|_{\varphi_iN_i(
J
)e_i^-})\notag\\
&\leq \op{Lip}(\varphi_i|_{(a_i^{-1}m_i)^p N_i(
J
)e_i^-})\op{Lip}((a_i^{-1}m_i)^p|_{N_i(
J
)e_i^-})\op{Lip}(\varphi_i^{-1}|_{\varphi_iN_i(
J
)e_i^-}).
\end{align*}
Since $\op{Lip}((a_i^{-1}m_i)^p|_{N_i(
J
)e_i^-})\to 0$ as $p\to\infty$ and $(a_i^{-1}m_i)^pN_i(
J
)e_i^-\to e_i^-$, 
we have $\op{Lip}(h_i^p|_{\varphi_iN_i(
J
)e_i^-})\to 0$ as $p\to \infty$.
Therefore the lemma follows.
\end{proof}

Since
$h_0^{p}\prod_{i=1}^{\mathsf r} n_iN_i(\eta_0^{t_i})e^-\to y_{h_0}$ uniformly, as $p\to \infty$, and $y_{h_0}\in N e^-$,
by possibly increasing $p_0$ if necessary, we may assume that $p_0$ satisfies that for all $p \ge p_0$,
 \begin{align} 
&h_0^{p}\prod_{i=1}^{\mathsf r} n_iN_i(\eta_0^{t_i})\subset NLN^+;  \\
&\sup_{u\in N_{\eta_0} y_{h_0}} 
 |\op{Jac}_u b^N(h_0^{p},\cdot)|\le 1/2; \label{eq.l2}\\ \label{eq.r1}
 & h_0^{p} N_{r}y_{h_0}\subset N_{r/2}y_{h_0}\quad 
  \text{for all } 0<r <  \eta_0.
  \end{align}

We make use the following simple observation:
\begin{lem}\label{p0}
If  there exists $p_1\ge 1$ such that
$$ \{\la (h_0^p):p\ge p_1\} \subset \op{Stab}_G([\mu]),$$
then $\la(h_0)\in \op{Stab}_G([\mu]). $
\end{lem} 

\begin{proof} Since $\op{Stab}_G([\mu])$ is a group and $\la(h_0)^p=\la(h_0^p)$,
the above lemma implies
that $$\la(h_0)= \la(h_0)^{p+1} \la(h_0)^{-p}\in \op{Stab}_G ([\mu]).$$ 
\end{proof}

Hence it suffices to show that for all $p\ge p_0$,
$\la (h_0^p)\in  \op{Stab}_G([\mu])$.
In the rest of this section, fix any $p\ge p_0$ and set
$$g_0=h_0^p.$$
We now assume that 
\be\label{con} \ell_0:=\la(g_0)\not\in\op{Stab}_G([\mu])\ee 
 and will prove that this assumption leads to a contradiction.

\medskip 

We write $g_i=h_i^p$ so that 
$$g_0=(g_1, \cdots, g_{\r}).$$
Noting
that $\varphi_i^-$ and
$\varphi_i^+$ are the attracting  fixed points of $g_i$ and $g_i^{-1}$ respectively, we
set $\varphi:=(\varphi_1, \cdots, \varphi_{\r})$. Hence
$\varphi^{\mp}=
(\varphi_1^{\mp}, \cdots, \varphi_{\r}^{\mp})$
are the attracting fixed points of $g_0^{\pm 1}$ respectively.
We set $$y_{g_0}:=\varphi^-.$$
 Note that $y_{g_0}=y_{h_0}$.
 By \eqref{contracting1}, for all $k\in \mathbb N$, we have
\be\label{contracting} \d_i(g_i^k z_i,g_i^k z_i')\le \frac{1}{2^{(t_i+1)k}}\cdot  \d_i(z_i,z_i')\ee  for any $z_i$, $z_i'\in\varphi_iN_i(J)e_i^-$.

We begin by presenting a long list of constants and subsets in a carefully designed order to be used in getting two contradictory upper and lower bounds in Lemmas \ref{lem.cov} and \ref{lem.cov3}. 

\subsection*{Definition of $E$, $\cal O_L$ and $\cal O_{\ell_0}$}
We fix subsets $E\subset \Ga\ba G$ and $\cal O_L\subset L$ as given by the following lemma:
\begin{lem} \label{OL}There exist an $N$-invariant  $\mu$-conull set $E\subset \Ga\ba G$ and a symmetric neighborhood $\cal O_L\subset L$ of $e$ such that 
$$E\cap E\ell_0^{-1}\cal O_L=\emptyset.$$
\end{lem}
\begin{proof} 
Since $\mu$ is $N$-ergodic and $\ell_0\notin \op{Stab}[\mu]$, $\mu$ and $\mu.\ell_0$ are mutually singular. Hence
there exists a $\mu$-conull subset $E'\subset \Ga\ba G$  with $E'\cap E'\ell_0=\emptyset$.  Let $c=1$ if $|\mu|=\infty$, and $c=|\mu|$ otherwise.
Choose $x\in E'\cap \text{supp}(\mu)$ and a bounded neighborhood $\cal O\subset G$ of $e$
such that $\mu(x\cal O)>c/2$. Set $F:=E'\cap x \cal O \ell_0^{-1}\cal O$.  Since $F\ell_0\subset E'\ell_0$ is a bounded null set,
there exists a symmetric neighborhood $\cal O_L\subset L\cap \cal O$ of $e$ such that
$\mu(F\cal O_L \ell_0)<c/4$. 
 Noting that $\mu(x\cal O-F\cal O_L \ell_0)>c/4$, we may choose a compact subset $C\subset x\cal O-F \cal O_L \ell_0$ with $\mu(C)>c/4$.
Since $C\ell_0^{-1}\cal O_L\subset x\cal O \ell_0^{-1}\cal O$,
we have $$C\ell_0^{-1}\cal O_L\cap E'\subset x\cal O \ell_0^{-1}\cal O\cap E'=F.$$
Since $C\ell_0^{-1} \cal O_L\cap F=\emptyset$ by the choice of $C$, we get $C\ell_0^{-1}\cal O_L\cap E'=\emptyset$ and hence
$\mu(C\ell_0^{-1}\cal O_L)=0$.
Consider the following $N$-invariant measurable subsets:
$$E_1:=\{z\in \Ga\ba G: \int_N \mathbbm{1}_C(zn)dn>0\} \quad \text{ and}$$
$$E_2:=\{z\in \Ga\ba G: \int_N\mathbbm{1}_{C\ell_0^{-1} {\cal O_L}}(zn)dn =0\}.$$
Recall $B_\v(j)$ denotes the set $\{n\in N: d_\v(n,e)< j\}$ for each $j\in \N$. Since
$\int_{z\in \Ga\ba G} \int_{B_\v(1)}\mathbbm 1_C(zn)dn d\mu(z)=\mu(C)m(B_\v(1)) >0$, we have
$\mu(E_1)>0$ by Fubini's lemma. 
Since $$ \int_{z\in \Ga\ba G} \int_{B_\v(j)}\mathbbm 1_{C\ell_0^{-1}\cal O_L}(zn)dn d\mu(z) =\mu(C\ell_0^{-1}\cal O_L) m(B_\v(j))=0,$$ again by Fubini's lemma,  $E_2(j)$ is $\mu$-conull, 
where $E_2(j):=\{z\in \Ga\ba G: \int_{B_\v(j)}\mathbbm{1}_{C\ell_0^{-1} {\cal O_L}}(zn)dn =0\}.$
Since $E_2=\cap_{j=1}^\infty E_2(j)$, the set $E_2$ is $\mu$-conull as well. Therefore, if we set $E=E_1\cap E_2$, then $E$ is an $N$-invariant measurable subset with $\mu(E)>0$. Now the $N$-ergodicity of $\mu$ implies that
$E$ is a $\mu$-conull subset.
Moreover,
we have $E\cap E\ell_0^{-1} {\cal O_L}=\emptyset$; to see this, suppose
$z=y\ell_0^{-1} \ell$ for some $z,y\in E$ and $\ell\in \cal O_L$.
Then
$\int_N\mathbbm{1}_{C\ell_0^{-1}\cal O_L}(y\ell_0^{-1} \ell n)dn =0$. By changing the variable $\ell_0^{-1}\ell n (\ell_0^{-1}\ell)^{-1}\to n $, it implies
that
$\int_N\mathbbm{1}_{C\ell_0^{-1}\cal O_L\ell^{-1}\ell_0}(y n)dn =0$. 
Since $C\subset C\ell_0^{-1}\cal O_L\ell^{-1}\ell_0$, we get $\int_N\mathbbm{1}_{C}(y n)dn =0$, implying $y\notin E$, yielding contradiction.
\end{proof}

We set \be\label{ell0} \cal O_{\ell_0}:=\ell_0 \cal O_L,\ee 
so that $E\cap E\cal O_{\ell_0}^{-1}=\emptyset$.

\medskip

For a differentiable map $f : N\to N$, let $\op{D}_uf : T_uN\to T_{f(u)}N$ denote the differential of $f$ at $u\in N$.
Let $\tau_u : N\to N$ denote the left translation map, i.e., $\tau_u(n)=un$ for $n\in N$. Choosing a basis $\cal B_e:=\{v_1,\cdots,v_m\}$ of $T_eN$, the collection $\cal B_w:=\{\op{D}_e \tau_w(v_1),\cdots, \op{D}_e \tau_w(v_m)\}$ gives a basis for $T_wN$ for each $w\in N$.
The following Jacobian of $f$ at $u\in N$ is well-defined, independent of the choice of $\cal B_e$:
$$
\op{Jac}_u f:=\det \, [\op{D}_u f]_{\cal B_{u}}^{\cal B_{f(u)}}.
$$
Here $[\op{D}_u f]_{\cal B_{u}}^{\cal B_{f(u)}}$ denotes the matrix representation of $\op{D}_uf$ with respect to the indicated bases.

\subsection*{Definition of $r_1, r_0$}  Since  $b^{AM}(g_0,y_{g_0})={\ell_0}$ and $b^{AM}(g_0,\cdot)$ is continuous at $y_{g_0}$, we can find $0<r_1<\min_i \frac{1}{2^{1+(1/t_i)}}\eta_0$ such that
%\begin{equation}\label{eq.rr}
%{\color{red}  N_{r_1}y_{g_0}\subset \prod_{i=1}^{\mathsf r} N_i(J)\varphi_i^-,\quad g_0\prod_{i=1}^{\mathsf r} n_iN_i(r_1^{t_i})\subset NLN^+\text{ and }   }
%\end{equation} 
 $$b^{AM}(g_0,N_{r_1}y_{g_0})\subset \cal O_{\ell_0}.$$

%{\color{red} Since $r_1<\eta_0$, it follows from %\eqref{contracting} that we have
% \begin{align}  \label{eq.l2}
% &\sup_{u\in N_{r_1} y_{g_0}} |\op{Jac}_u b^N(g_0^j,\cdot)|\le 1/2
% \quad\text{for all $j \ge 1$ ;}\\&
%  \label{eq.r1}  g_0N_{r_1}y_{g_0}\subset N_{r_1/2}y_{g_0} .\end{align}}
Set
$$
r_0:=\frac{3}{4} r_1.
$$
\subsection*{Definition of $ k, c,\eta$}
By \eqref{eq.J},
we have $g_0^j   N_{\eta_0} e^-\to y_{g_0}$ uniformly as $j \to\infty$.
Hence we may fix a large integer $k\ge 1$ which satisfies the following three conditions for all $1\le i\le \r$:
\begin{align} \label{eq.sqz}
& N_{r_1/2}y_{g_0}\subset N_{r_0}\,g_0^k N_{\eta_0} e^- \subset N_{r_1}y_{g_0};\\
& \label{k1}  b^{N_i} (g_i^k, N_i({\eta_0}^{t_i}))\subset n_i N_i (r_0^{t_i}/4);\\&
g_0 b^N(g_0^k,N_{\eta_0})N_{r_0} \subset NLN^+
\end{align}
where $n_i$ is given in \eqref{ni}.
Since $g_i^k e_i^+\ne g_i^k e_i^-$,
we can choose $0<\eta<\frac{1}{2} \eta_0$ satisfying
\be\label{etachoice} g_i^k e_i^+\not\in b^{N_i}(g_i^k, 
e_i
)N_i(\eta^{t_i}) e_i^-\quad\text{ for all $i$}.\ee
We fix a small number $0<c<1/2$ so that for all $1\le i\le \r$ and $x, y\in N_i(\eta^{t_i})e_i^-$,
\be\label{cfix} (2c)^{t_i} \d_i(x,y) \le \d_i( b^{N_i}(g_i^k, x),
b^{N_i}(g_i^k, y))\ee
and
$$2c<\min  \left( \inf_{u\in N_{r_1}}|\op{Jac}_u b^N(g_{0},\cdot)|, \inf_{u\in N_{r_1}} |\op{Jac}_u b^N(g_{0}^{k},\cdot)|\right) .$$

\begin{lemma}\label{eq.Jac0}
We have
\be\label{first}
 b^N(g_0^k,e)N_{2c\eta}\subset b^N(g_0^k,N_\eta) \subset b^N(g_0^k,e)N_{\eta},
\ee
and 
\be\label{second}
b^N(g_0^k,N_{c\eta})\subset b^N(g_0^k,e)N_{c\eta}.
\ee
\end{lemma}
\begin{proof} Fix $1\le i\le \r$.
By \eqref{etachoice}, we have
$b^{N_i}(g_i^k, 
e_i
)N_i(\eta^{t_i}) e_i^-\subset g_i^kN_i e_i^-$ and hence
 $b^{N_i}(g_i^k, e_i
)N_i(\eta^{t_i})\subset b^{N_i}(g_i^k, N_i)$. 
Let $n\in N_i((2c\eta)^{t_i})$ be arbitrary.  There exists $n'\in N_i$  such that $b^{N_i}(g_i^k,e_i
)n= b^{N_i}(g_i^k,n')$.
We have, by \eqref{cfix},
\begin{align*}
    (2c)^{t_i}\mathsf d_i(e_i
,n') &\leq \mathsf d_i(b^{N_i}(g_i^k,
e_i
), b^{N_i}(g_i^k,n'))\\
    &=\mathsf d_i(b^{N_i}(g_i^k,e_i
), b^{N_i}(g_i^k,e_i
)n)=\mathsf d_i(e_i
,n)\leq(2c \eta)^{t_i}
\end{align*}
and hence $\d_i(e_i
, n')\le \eta^{t_i}$. It implies
$$b^{N_i}(g_i^k,e_i
)n=
b^{N_i}(g_i^k,n')\in b^{N_i} (g_i^k, N_i(\eta^{t_i})).$$
This proves the first inclusion in \eqref{first}. 

By \eqref{contracting} and \eqref{eq.J}, we have
$$\d_i (g_i^kne_i^-, g_i^k n'e_i^- )\le 2^{-k} \d_i(ne_i^-, n'e_i^-) \;\; \text{ for all $n, n'\in N_i(\eta^{t_i})$.}$$

In other words, for all $n,n'\in N_i(\eta^{t_i})$,
\be\label{qqq} \d_i ( b^{N_i}(g_i^k,n) ,b^{N_i}(g_i^k,n'))\le  2^{-k}\d_i(n, n'). \ee

Hence $b^{N_i}(g_i^k,\cdot)$ has Lipschitz constant less than $1$ on $N_i(\eta^{t_i})$, the right inclusion in \eqref{first}, as well as \eqref{second} follow.
\end{proof}

\begin{lem} \label{bbb} We have \be
b^N(g_0,b^N(g_0^k,v)N_{r_0}) \subset  b^N(g_0^k,v)N_{r_0} \quad \text{ for all }v\in N_\eta.\label{eq.er2}
\ee \end{lem}
\begin{proof}

As $\d_i$ is left-invariant, the choice of $k$ as in \eqref{k1} implies that for any $v\in N_i(\eta^{t_i})$, we have
$$b^{N_i} (g_i^k, v)N_i(r_0^{t_i})\supset n_i N_i(3 r_0^{t_i}/4) \quad\text{ and} $$
$$b^{N_i}(g_i^k, N_i({\eta}^{t_i})) N_i({r_0}^{t_i}) \subset n_i N_i(3 r_0^{t_i}/2).$$

Since $r_1<\min_i \frac{1}{2^{1+(1/t_i)}}\eta_0$ and hence $3 r_0^{t_i}/2 <\eta_0^{t_i}$ by the definition of $r_0$, it follows from \eqref{contracting} and the property $g_i\varphi_i^-=\varphi_i^-$
that 
$$g_in_i N_i(3 r_0^{t_i}/2) \subset n_i N_i(3r_0^{t_i}/4).$$

Therefore, for any $v\in N_i(\eta^{t_i})$,
$$b^{N_i}(g_i,b^{N_i}(g_i^k,v)N_i({r_0}^{t_i})) \subset {n_i}N_i(3 r_0^{t_i}/4)\subset b^{N_i} (g_i^k, v)N_i(r_0^{t_i}). $$
This proves the lemma. 
\end{proof}

\subsection*{Definition of $V_0$.}
Since the following \eqref{eq.r1'}- \eqref{eq.gk3} are all open conditions which have been proved at $g=g_0$ in  \eqref{eq.l2},\eqref{eq.r1}, \eqref{eq.sqz} and Lemmas \ref{eq.Jac0} and \ref{bbb},
we may choose a bounded neighborhood $V_0$ of $g_0$ in $G$ such that those conditions
continue to hold for all $g\in V_0$, $u\in N_{r_0}b^N(g^k,N_\eta)$ and $v\in N_\eta$:
\begin{align}
&gN_{r_1}y_{g_0}\subset N_{r_1/2}y_{g_0},\label{eq.r1'}\\
&N_{r_1/2}y_{g_0}\subset N_{r_0}\,g^k N_\eta e^- \subset N_{r_1}y_{g_0},\label{eq.sqz'}\\
& b^{AM}(g,u)\in\cal O_{\ell_0},\label{eq.l'}\\
& b^N(g^k,e)N_{2c\eta }\subset b^N(g^k,N_\eta)\subset b^N(g^k,e)N_{\eta}\quad \text{ and }\label{eq.er}\\
&b^N(g^k,N_{c\eta})\subset b^N(g^k,e)N_{c\eta}.
\\ \label{eq.Jac}
&2c <|\op{Jac}_u b^N(g,\cdot)|<1,\quad 2c <|\op{Jac}_v b^N(g^k,\cdot)|<1\\
& b^N(g,b^N(g^k,v)N_{r_0}) \subset  b^N(g^k,v)N_{r_0}\label{eq.er2.new}\\
& g b^N(g^k,N_\eta)N_{r_0} \subset NLN^+ .\label{eq.gk3}
\end{align}

\subsection*{Definition of $R$, $\cal B_L$, and $\cal B_{N^+}$}
Since the sets  $V_0$, $N_\eta$ and $\{b^N(g^k,N_\eta)N_{r_0} :g\in V_0\}$ are bounded, it follows from \eqref{eq.gk3} that there exist $R>0$ and bounded symmetric neighborhoods
$ \cal B_L\subset L$ and $\cal B_{N^+}\subset N^+$ of $e$ such that for all $g\in V_0$,
\be\label{eq.RRR}
g^k N_\eta\subset N_R \cal B_L \cal B_{N^+}\quad \text{ and}\quad  g b^N(g^k,N_\eta)N_{r_0}\subset N_R {\cal B_L} \cal B_{N^+}.
\ee

\subsection*{Definition of $\beta$, $R'$ and $\kappa_*$} 
We fix $\beta>0 $ such that 
\be\label{betad} a_t^{-1}N_R N_\eta a_t N_{c\eta} \subset N_{2c\eta}  \quad\text{for all $t\ge \beta$.}\ee 
We also fix $R'>0$ so that
\be\label {Rprime} \bigcup_{t\in [-\beta, \beta]} N_RN_\eta (a_{t} N_{\eta} N_{R}N_{ c\eta}a_{t}^{-1})\subset N_{R'}.\ee 

Recalling the notation from Lemma \ref{ka1}, we set
\be\label{ka}\kappa_*:=\kappa_* (\v,\beta, c\eta, R') =\frac{m(N_{R'})}{ m(N_{c\eta})} {\kappa_{\v}}  e^{\|2\rho\|\beta}.\ee

\subsection*{Definition of $\Om$, $\tOm$, $\cal O_{N^+}$,  $Q$, $Q_\bot$ and $T_0$}
Let $E$ be an $N$-invariant $\mu$-conull set as in Lemma \ref{OL}.
We fix a compact subset $\Omega\subset E$ with $\mu(\Omega)>0$, and define
\begin{equation}\label{eq.def_tOm}
	\tOm:= \Om \cal B_L \cal B_{N^+}.
\end{equation}
Since $ \mu(\tOm)=\mu(\tOm \cap E) $, we can find a compact set $\Om \subset Q \subset \tOm \cap E $
satisfying
\begin{equation}\label{eq.mu(Q)}
	\mu(\tOm - Q) < \frac{c}{16\kappa_{0} \kappa_*}.
\end{equation}
Since $ Q \subset E $, we know $ \mu(Q\cal O_{\ell_0}^{-1}) = 0 $.
By the uniform convergence theorem, there exists a bounded symmetric neighborhood $ \cal O_{N^+} \subset \cal B_{N^+} $ of $ e $ for which the set
\begin{equation}\label{eq.defQbot}
	Q_\bot := Q \cal O_{N^+}\cal O_{\ell_0}^{-1}
\end{equation}
satisfies
\begin{equation}\label{eq.mu(Qbot)}
	\mu(Q_\bot) < \frac{c^2}{16 \kappa_{\v} \kappa_{0} \kappa_*} \mu(\Omega).
\end{equation}

We fix $T_0>0$ such that
\begin{equation}\label{eq.T1x}
\op{Ad}_{a_t} \cal B_{N^+}\subset {\cal O}_{N^+}\quad \text{ for all }  t\ge T_0.
\end{equation}

\subsection*{Definition of $T_1$, $\Om_1$, $ \Om_2 $, $ \Xi $ and $\Theta$} Since $\cal S_x(\v)=\cal S_\mu(\v)$ for $ \mu $-a.e.~$ x \in \Gamma \ba G $, we can  find  $T_1>T_0$ so that the set 
\begin{equation}\label{eq.def_tOm1}
	\tOm_1 := \{x\in \tOm:  \op{Stab}_G(xa_t)\cap V_0\neq\emptyset\text{ for some $T_0\le t\le T_1$}\}
\end{equation}
satisfies
\begin{equation}\label{eq.Om2x}
 \mu (\tOm - \tOm_1) < \frac{1}{4}  \mu(\Omega).
\end{equation}
Set
\begin{equation}\label{eq.def_Om1}
	\Om_1:=\Om \cap \tOm_1.
\end{equation}

Since $ \Om \subset \tOm $, we have
\begin{equation}\label{eq.mu(Om1)}
	\mu(\Om_1) \geq \mu(\Om) - \mu (\tOm - \tOm_1) > \frac34 \mu(\Om).
\end{equation}

We define
\begin{equation}\label{eq.defXi}
	\Xi:=\left\{x\in\Gamma \ba G: \exists t>0 \text{ s.t } 
		{ \int_{a_t N_{ r_0}a_t^{-1}}\mathbbm{1}_{Q_\bot}(x n)\,dn}\geq 2c {\int_{a_t N_{r_0}a_t^{-1}}\mathbbm{1}_{Q}(x n)\,dn}
	\right\}.
\end{equation}
Set
\begin{equation}\label{eq.defOm2}
	\Om_2 := \Om_1 -\Xi.
\end{equation}

Recall the notation for distance $d_0$ on $ N $ and the corresponding metric balls $B_0(r)$, $r>0$, from Proposition \ref{REH}.
Consider the following set
\begin{equation}\label{eq.defTheta}
	\Theta:=\left\{x\in \Gamma \ba G: \exists r>0 \text{ s.t }
		{ \int_{B_0(r)}\mathbbm{1}_{\tOm \cap \Xi}(x n)\,dn}\geq \frac{c}{\kappa_*} {\int_{B_0(r)}\mathbbm{1}_{\Om_2}(x n)\,dn}
	\right\}.
\end{equation}

\begin{prop}\label{half}
We have 
$$
	\mu(\Om_2-\Theta) > \frac{1}{4}\mu(\Om ) .$$
\end{prop}
\begin{proof} Since $a_tN_{r_0}a_{t}^{-1}=B_\v(e^t r_0)$ for any $t, r_0>0$, we may apply the maximal ratio inequality (Lemma \ref{maxi}) and \eqref{eq.mu(Qbot)} and get
$$\mu(Q\cap \Xi)\le \frac{2\kappa_{\v}}{2c} \mu(Q_\bot)  < \frac{\kappa_{\v}}{c}\cdot \frac{c^2}{16 \kappa_{\v} \kappa_{0} \kappa_*} \mu(\Om)  = \frac{c}{16\kappa_{0} \kappa_*}\mu(\Omega).$$
Therefore,
by \eqref{eq.mu(Q)},
\[ \mu(\tOm \cap \Xi) \leq \mu(\tOm - Q) + \mu (Q \cap \Xi) < \frac{c}{8 \kappa_{0} \kappa_*}\mu(\Om). \]
By \eqref{eq.mu(Om1)}, we have
\[ \mu(\Om_2)=\mu(\Om_1-\Xi) \geq \mu(\Om_1)-\mu(\tOm \cap \Xi) \geq \left(\frac34-\frac{c}{8 \kappa_{0} \kappa_*}\right)\mu(\Om) > \frac12 \mu(\Om). \]

Employing the maximal ratio inequality yet again, we deduce
\[ \mu(\Om_2 \cap \Theta) \leq \frac{2\kappa_{0} \kappa_*}{c} \mu(\tOm \cap \Xi) < \frac{2\kappa_{0} \kappa_*}{c}\cdot \frac{c}{8\kappa_{0} \kappa_*}\mu(\Om) = \frac14 \mu(\Om), \]
implying the claim by \eqref{eq.mu(Om1)}.
\end{proof}

\subsection*{Choice of $x_0$, $R_1$, $R_2$ and $D$.}
We fix $R_1, R_2>0$ so that $N_R\subset B_0(R_1)$ and
\be\label{rdouble}\bigcup_{0<t\le T_1}a_t B_0(R_1)a_t^{-1}\subset B_0(R_2) .\ee
We choose $x_0$ and $D$ as in the following lemma:
\begin{lemma}\label{lem.x0}
There exist $x_0\in\Gamma \ba G$ and a ball $D=B_0(R_{x_0})$ with $ R_{x_0}>R_2 $ such that
\begin{equation*}\frac{\int_{ D}\mathbbm{1}_{\tOm \cap \Xi}(x_0n)\,dn}{\int_{D}\mathbbm{1}_{\Om_2}(x_0n)\,dn}<\frac{c}{\kappa_*}, \quad \text{and} \quad
\frac{\int_{\partial_{R_2} D}\mathbbm{1}_{\Om_2}(x_0n)\,dn}{\int_{D}\mathbbm{1}_{\Om_2}(x_0n)\,dn}<\frac{1}{2}
\end{equation*}
where $\partial_r B_0(R_{x_0}):=B_0(R_{x_0})-B_0(R_{x_0}-r)$.
\end{lemma}
\begin{proof} Choose any $x_0\in \Om_2-\Theta$, which is possible by Proposition \ref{half}.
By the definition of $ \Theta $, $x_0$ satisfies the first inequality for any ball $D=B_0(R) $. By Lemma \ref{lem.x1}, there exists $ R_{x_0}>R_2 $ satisfying the second inequality, as required.
\end{proof}

 For any $X\subset\Ga\ba G$, define the subset $\T_{X}\subset N$  by
$$
\T_{X}:=\{n\in N : x_0n\in  X\}.
$$

\subsection*{Definition of $t_u, a_u, g_u$} By the definition of $\Om_1$ in \eqref{eq.def_tOm1},
for each $u\in \T_{\Om_1}$, we can choose $T_0\le t_u\le T_1$ such that 
$$ \op{Stab}_G(x_0u   a_{t_u})\cap V_0\neq\emptyset.$$
We set $a_u:=a_{t_u}$ for the sake of simplicity, and choose
$$g_u\in \op{Stab}_G(x_0ua_u)\cap V_0.$$

\begin{lemma}\label{lem.P1x}
For  $u\in \T_{\Om_1}$, we have $ua_ub^N(g_u^k,N_\eta)a_u^{-1}\subset \T_{\Xi}$.
\end{lemma}
\begin{proof}
Let $u\in\T_{\Om_1}$ and $v_0\in N_\eta$ be arbitrary. Setting $v_0':=b^N(g_u^k,v_0)$,
we need to show that $x_0ua_uv_0'a_u^{-1}\in \Xi$.
Observe that for all $v\in N$,
\begin{align}\label{eq.BH0}x_0u(a_uva_u^{-1}) &=x_0ua_u(g_uv)a_u^{-1}\\&=x_0ua_u( b^N(g_u,v)\,b^{AM}(g_u,v)\,b^{N^+}(g_u,v))a_u^{-1}\notag\\&=x_0u(a_u b^N(g_u,v)a_u^{-1})\,  b^{AM}(g_u,v) \,(a_ub^{N^+}(g_u,v)a_u^{-1}),\notag
\end{align}
whenever $b(g_u,v)$ is defined.
For any $n\in N_{r_0}$, we can plug $v=v_0'n$
into \eqref{eq.BH0} by \eqref{eq.gk3}, and
 get
\begin{align*}
x_0u(a_uv_0'na_u^{-1})
&=x_0u(a_u b^N(g_u,v_0'n)a_u^{-1})(\ell   a_u b^{N^+}(g_u,v_0'n)a_u^{-1})
\end{align*}
where $\ell:=b^{AM}(g_u,v_0'n)\in \cal O_{\ell_0}$  by \eqref{eq.l'}.

Recall that $b^{N^+}(g_u,v_0'n)\in \cal B_{N^+}$ by \eqref{eq.RRR} and 
  $\op{Ad}_{ a_t}(\cal B_{N^+})\subset {\cal O}_{N^+}$ for all $  t\ge T_0$ by \eqref{eq.T1x}.
  It follows that
  $$
a_u b^{N^+}(g_u,v_0'n)a_u^{-1}\in {\cal O}_{N^+}.
$$
  
Since  
$$x_0u(a_u b^N(g_u,v_0'n)a_u^{-1})=
x_0u(a_uv_0'na_u^{-1}) (  a_u b^{N^+}(g_u,v_0'n)a_u^{-1})^{-1} \ell^{-1} ,$$
and $ Q_\bot = Q \cal O_{N^+}\cal O_{\ell_0}^{-1} $ as defined in \eqref{eq.defQbot},
we have for all $n\in N_{r_0}$,
\begin{equation}\label{eq.perpx}
\mathbbm{1}_{Q}(x_0ua_uv_0'  na_u^{-1})\leq \mathbbm{1}_{Q_\bot}(x_0u(a_u b^N(g_u,v_0'n)a_u^{-1})).
\end{equation}

Note that 
\begin{align*}
&\int_{N_{r_0}}\mathbbm{1}_{Q}(x_0ua_uv_0'a_u^{-1}(a_una_u^{-1}))\,dn\\
&\leq \int_{N_{r_0}}\mathbbm{1}_{Q_\bot}(x_0ua_u b^N(g_u,v_0'n)a_u^{-1})\,dn \quad \text{by \eqref{eq.perpx}} \\
&\leq (2c)^{-1}\int_{b^N(g_u,v_0'N_{r_0})}\mathbbm{1}_{Q_\bot}(x_0u(a_u na_u^{-1}))\,dn\;\;
\text{by \eqref{eq.Jac} and Lemma \ref{JJac}}
\\
&\leq (2c)^{-1}\int_{v_0'N_{r_0}}\mathbbm{1}_{Q_\bot}(x_0u(a_u na_u^{-1}))\,dn\quad \text{by \eqref{eq.er2.new}}\\
&= (2c)^{-1}\int_{N_{ r_0}}\mathbbm{1}_{Q_\bot}(x_0ua_uv_0'a_u^{-1}(a_u na_u^{-1}))\,dn.
\end{align*}

Hence by the change of variable formula, we have
\begin{equation*}\label{eq.cofv}
\int_{a_u N_{r_0}a_u^{-1}}\mathbbm{1}_{Q_\bot}(x_0ua_uv_0'a_u^{-1} n)\,dn\geq 2 c \int_{a_u N_{r_0}a_u^{-1}}\mathbbm{1}_{Q}(x_0ua_uv_0'a_u^{-1}  n)\,dn.
\end{equation*}
In view of definition \eqref{eq.defXi}, this proves that $x_0ua_uv_0'a_u^{-1} \in\Xi $.
\end{proof}

Although the following lemma, which was used in the above proof,
should be a standard fact, we could not find a reference, so we provide a proof.
\begin{lem}\label{JJac}
For any measurable function $f:N\to\bb R$ and a differentiable map $\phi : N\to N$, we have
$$
\int_N (f\circ\phi)(n)\,|\op{Jac}_n\phi|\,dn=\int_N f(n)\,dn.
$$
\end{lem}
\begin{proof}
Since $N$ is a simply connected nilpotent Lie group, the Haar measure $dn$ on $N$ is the push-forward of the Lebesgue measure $d\op{Leb}$ on $\frak n=\op{Lie} N$ by the exponential map.
Let  $\tilde\phi:=\log\circ\,\phi\circ\exp$.
Note that $\op{Id}+\frac{1}{2}\op{ad}_{x}\in\op{GL}(\frak n)$ is unipotent for all $x\in\frak n$, as $\op{ad}_{x}\in \op{End}(\frak n)$ is a nilpotent element.  
We claim that $|\op{Jac}_{e^x}\phi|=|\op{Jac}_{x}\tilde \phi|$.

Since $N$ is a nilpotent Lie group of at most $2$-step, 
we have for any $n, n'\in N$,$$
\log(nn')=\log n+\log n'+\frac{1}{2}[\log n,\log n'].
$$
Hence, we get via a direct computation:
\begin{align*}
&\frac{d}{dt}\log \phi(e^x)^{-1}\phi(e^{x}e^{ty})\\
&=\frac{d}{dt}\log \phi(e^x)^{-1}\phi(e^{x+{t}y+\frac{1}{2}t[x,y]})\\
&=\frac{d}{dt}\left(\log  \phi(e^x)^{-1}+\log \phi(e^{x+{t}y+\tfrac{1}{2}t[x,y]})+\tfrac{1}{2}[ \log  \phi(e^x)^{-1},\log \phi(e^{x+{t}y+\tfrac{1}{2}t[x,y]})]\right)\\
&=(\op{Id}_{\frak n}+\tfrac{1}{2}\op{ad}_{-\tilde\phi(x)})\left(\frac{d}{dt}\tilde\phi(x+t(y+\tfrac{1}{2}[x,y]))\right)\\
&=(\op{Id}_{\frak n}+\tfrac{1}{2}\op{ad}_{-\tilde\phi(x)})\circ (\op{D}_x\tilde\phi)(y+\tfrac{1}{2}[x,y]).
\end{align*}
Now let $x\in\frak n$ and $y\in T_{e^x}N$.
In view of the identification $\frak n=T_eN\simeq T_{n}N$ for $n=e^x$ and $\phi(e^x)$, we have
\begin{align*}
&\op{D}_{e^x}\phi(y)=\frac{d}{dt}\bigg\rvert_{t=0}\phi(e^x)^{-1}\phi(e^xe^{ty})\\
&=\frac{d}{dt}\bigg\rvert_{t=0}\exp\circ\log \phi(e^x)^{-1}\phi(e^xe^{ty})\\
&=(\op{D}_0\exp)\left(\frac{d}{dt}\bigg\rvert_{t=0}\log \phi(e^x)^{-1}\phi(e^xe^{ty})\right)\\
&=(\op{D}_0\exp)\circ (\op{Id}_{\frak n}+\tfrac{1}{2}\op{ad}_{-\tilde\phi(x)})\circ (\op{D}_x\tilde\phi)(y+\frac{1}{2}[x,y])\\
&= (\op{D}_0\exp)\circ (\op{Id}_{\frak n}+\tfrac{1}{2}\op{ad}_{-\tilde\phi(x)})\circ (\op{D}_x\tilde\phi)\circ (\op{Id}_{\frak n}+\tfrac{1}{2}\op{ad}_{x})
(y)
\end{align*}
where we have used the convention $\frac{d}{dt}|_{t=0}\beta\in T_{\beta(0)}N$ to denote the element of $T_{\beta(0)}N$ represented by a smooth curve $\beta : (-\e,\e)\to N$.
Since $\op{D}_0\exp : T_0\frak n\to T_eN=\frak n$ is the identity map $\op{Id}_{\frak n}$ under the identification $T_0\frak n\simeq\frak n$
and $\op{Id}_{\frak n}+\tfrac{1}{2}\op{ad}_{z}:\frak n\to \frak n$ 
has determinant one for any $z\in \frak n$, being a unipotent matrix,
we deduce that $\text{det}( \op{D}_{e^x}\phi)= \text{det} (\op{D}_x \tilde \phi)$, proving the claim.
Hence for any measurable function $f:N\to\bb R$, we have
\begin{align*}
&\int_N (f\circ\phi)(n)\,|\op{Jac}_n\phi|\,dn=\int_{\frak n}(\tilde f\circ\tilde\phi)(x)\,|\op{Jac}_{e^x}\phi|\,d\op{Leb}(x)\\
&=\int_{\frak n} (\tilde f\circ\tilde\phi)(x)\,|\op{Jac}_{x}\tilde \phi|\,d\op{Leb}(x)=\int_{\frak n} \tilde f(x)\,d\op{Leb}(x)=\int_N f(n)\,dn,
\end{align*}
where we have used the change of variable formula for the Lebesgue measure in the second last equality.
This proves the lemma.
\end{proof}

\subsection*{Definition of $B_u$, $J_u$}
For each $u\in \T_{\Om_1}$, we define
\begin{align*}
B_{u}&:=ua_{u} N_{c\eta} a_{u}^{-1},\text{ and }\\
J_{u}&:=\{ua_ub^N(g_u^k,n)a_u^{-1}:n\in N_{c\eta}, x_0ua_{u}na_{u}^{-1}\in\Om\}.
\end{align*}

\begin{lemma}\label{lem.asy}
For all $u\in\T_{\Om_1}$, we have 
$$
2c\cdot  m(B_u\cap\T_{\Om})\le m(J_u).
$$
\end{lemma}
\begin{proof}
Defining $\varphi_u : N\to N$ by $\varphi_u(n)=u (a_{u}na_{u}^{-1})$, we have
\begin{align*}
&J_{u}=(\varphi_u\circ b^N(g_{u}^k,\cdot)\circ\varphi_u^{-1})(B_{u}\cap\T_\Omega).
\end{align*}
For all $v\in {N_{c\eta} \subset} N_\eta$, we have $2c \le |\op{Jac}_v b^N(g_{u}^k,\cdot)|$ by \eqref{eq.Jac}, 
and hence 
$$
2c \le |\op{Jac}_v(\varphi_u\circ b^N(g_{u}^k,\cdot)\circ\varphi_u^{-1})|.
$$
The lemma follows from Lemma \ref{JJac}.
\end{proof}

\begin{lemma}\label{lem.chi}
For any $u\in\T_{\Om_1}\cap (D-\partial_{R_2} D)$, we have
$$J_u\subset \T_{\tOm \cap \Xi}\cap D.$$

%\begin{enumerate}
%\item For any $u\in\T_{\Om_1}$, we have
%$$J_u\subset\T_{\tOm}.$$ 
%\item For any $u\in\T_{\Om_1}\cap D-\partial_{R_2} D$, we have
%$$J_u\subset \T_{\tOm}\cap D.$$ \end{enumerate}
\end{lemma}
\begin{proof}
Let $u\in\T_{\Om_1}$ and $v\in J_u$ be arbitrary.
Then $v=u(a_u b^N(g_u^k,n)a_u^{-1})$ for some $n\in N_{c\eta}$. Since  $x_0u\in\Om_1$ we have for all $n\in N_{c\eta}$,
\begin{align}\label{eq.BH}
&x_0u(a_una_u^{-1})=x_0ua_u(g_u^kn)a_u^{-1}\\
&=x_0ua_u( b^N(g_u^k,n)\,b^{AM}(g_u^k,n)\,b^{N^+}(g_u^k,n))a_u^{-1}\notag\\
&=x_0u(a_u b^N(g_u^k,n)a_u^{-1})\,  b^{AM}(g_u^k,n) \,(a_ub^{N^+}(g_u^k,n)a_u^{-1}),\notag
\end{align}
with $b^{AM}(g_u^k,n)\in  {\cal B_L}$ and $b^{N^+}(g_u^k,n)\in \cal B_{N^+}$, by \eqref{eq.RRR}.
Since $t_u\ge T_0$, we have $a_ub^{N^+}(g_u^k,n)a_u^{-1}\subset {\cal O}_{N^+}$ by \eqref{eq.T1x}.
Hence, 
\begin{align*}
	x_0v&=x_0u(a_u b^N(g_u^k,n)a_u^{-1})
\\&= x_0u (a_u n a_u^{-1}) (a_ub^{N^+}(g_u^k,n)^{-1} a_u^{-1}) b^{AM}(g_u^k,n)^{-1}
\in \Om   {\cal O}_{N^+} {\cal B_L}.
\end{align*}
Since $ \cal O_{N^+} \subset \cal B_{N^+} $ we deduce
\[ x_0 v \in \Om {\cal B}_{N^+} {\cal B_L} \subset \tOm. \]
By Lemma \ref{lem.P1x}, since $ v=u(a_u b^N(g_u^k,n)a_u^{-1}) $ with $ u \in \Om_1 $ and $ n \in N_{c\eta} \subset N_\eta $ we have
$ x_0 v \in \Xi $ implying $ v \in \T_{\tOm \cap \Xi} $.

Further assuming that $u\in D-\partial_{R_2} D$, since $b^N(g_u^k,n)\in N_R\subset B_0(R_1)$, by \eqref{eq.RRR},
it follows from \eqref{rdouble} that $$
a_u b^N(g_u^k,n)a_u^{-1}\in B_0({ R_2 }).
$$
Since $d_0$ is a distance, satisfying the triangle inequality,  we deduce that $v\in D$, as claimed.
\end{proof}

\subsection*{Properties of coverings}

For all $u\in\T_{\Om_1}$, we have
\begin{align}\label{eq.crrx}
&b^N(g_u^k,e)N_{2c\eta }\subset b^N(g_u^k,N_\eta)\subset b^N(g_u^k,e)N_{\eta} \text{ and }\\
&b^N(g_u^k,N_{c\eta})\subset b^N(g_u^k,e) N_{c\eta}.\notag
\end{align}
Setting \be\label{wu} w_u:=ua_ub^N(g_u^k,e)a_u^{-1},\ee
we have
\begin{align}\label{eq.wix}
&J_u
\subset w_u a_uN_{c\eta}a_u^{-1}\text{ and }\\
&w_u a_uN_{2c\eta}a_u^{-1}\subset  ua_ub^N( g_u^k,N_{\eta})a_u^{-1}\subset w_u a_uN_\eta a_u^{-1}.\notag
\end{align}
Since $b^N(g_u^k,e)\in N_R$ by \eqref{eq.RRR}, we have $w_u\in  u a_uN_Ra_u^{-1}$.
Hence
\begin{equation}\label{eq.uix}
J_u\subset w_u a_uN_{2c\eta}a_u^{-1}\subset ua_ub^N( g_u^k,N_\eta)a_u^{-1}\subset   u a_uN_RN_\eta a_u^{-1}.
\end{equation}

\begin{lemma}\label{lem.123x}\label{lem.bddx} 
If $u_i,u_j\in \T_{\Om_2}$ satisfy that $J_{u_i}\cap J_{u_j}\neq\emptyset$, then
\begin{enumerate}
\item
$a_{u_i} ^{-1}a_{u_j}N_RN_\eta a_{u_j}^{-1}a_{u_i}N_{ c\eta}\not\subset N_{2c\eta}$,
\item
$u_i^{-1}u_j\in a_{u_i}N_RN_\eta a_{u_i}^{-1}a_{u_j} N_\eta N_R a_{u_j}^{-1}$, 
\item
$B_{u_j}\subset u_i a_{u_i}N_RN_\eta (a_{u_i}^{-1}a_{u_j} N_{\eta} N_{R} N_{ c\eta}a_{u_j}^{-1}a_{u_i})a_{u_i}^{-1}$, and
\item $a_{u_i}^{-1} a_{u_j}\subset \exp( [-\beta, \beta]\v )$.
\end{enumerate}
\end{lemma}
\begin{proof}
To prove (1), let $v\in J_{u_i}\cap J_{u_j}$.
By \eqref{eq.uix}, we have $u_j^{-1}v\in a_{u_j}N_RN_\eta a_{u_j}^{-1}$ and by \eqref{eq.wix}, we have $v^{-1}w_{u_i}\in a_{u_i}N_{ c\eta}a_{u_i}^{-1}$, using the fact that $N_{c\eta}$ is symmetric.
Hence,
\begin{equation}\label{eq.uwx}
u_j^{-1}w_{u_i}=(u_j^{-1}v)(v^{-1}w_{u_i})\in a_{u_j}N_RN_\eta a_{u_j}^{-1}a_{u_i}N_{ c\eta}a_{u_i}^{-1}.
\end{equation}
Since $u_i\in\T_{\Om_1}$ and $u_j\not\in\T_{\Xi}$, we have $u_j\not\in u_ia_{u_i}b^N( g_{u_i}^k,N_\eta)a_{u_i}^{-1}$ by Lemma \ref{lem.P1x}.
It follows from \eqref{eq.wix} that $u_j\not\in  w_{u_i}a_{u_i}N_{2c\eta} a_{u_i}^{-1}$, or equivalently,
$$
u_j^{-1}w_{u_i}\not\in a_{u_i}N_{2c\eta} a_{u_i}^{-1}.
$$
Note that by \eqref{eq.uwx},
$$
a_{u_i}^{-1}u_j^{-1}w_{u_i}a_{u_i}\in a_{u_i}^{-1}a_{u_j}N_RN_\eta a_{u_j}^{-1}a_{u_i}N_{ c\eta}- N_{2c\eta},
$$
proving (1).
We now prove (2).
Since $J_{u_i}\cap J_{u_j}\neq\emptyset$, by \eqref{eq.wix} and \eqref{eq.uix}, 
$$
u_i a_{u_i}N_RN_\eta a_{u_i}^{-1}\cap u_j a_{u_j}N_RN_\eta a_{u_j}^{-1}\neq\emptyset.
$$
Since $N_\eta$ and $N_R$ are symmetric, we get
$$
u_i^{-1}u_j\in a_{u_i}N_RN_\eta a_{u_i}^{-1}a_{u_j} N_\eta N_R a_{u_j}^{-1}.
$$

To check (3), observe that
\begin{align}\label{eq.UB}
B_{u_j}&=u_j a_{u_j}N_{ c\eta}a_{u_j}^{-1}=u_i(u_i^{-1}u_j)a_{u_j}N_{ c\eta}a_{u_j}^{-1}\\
&\subset u_i(a_{u_i}N_RN_\eta a_{u_i}^{-1}a_{u_j} N_\eta N_R a_{u_j}^{-1})a_{u_j}N_{c\eta}a_{u_j}^{-1}\notag\\
&= u_ia_{u_i}N_RN_\eta (a_{u_i}^{-1}a_{u_j} N_\eta N_R N_{c\eta}a_{u_j}^{-1}a_{u_i})a_{u_i}^{-1},\notag
\end{align}
where the inclusion $\subset$ follows from Claim (2).
Claim (4) follows from (1) by the  choice of $\beta$ as in \eqref{betad}.
\end{proof}

\begin{lemma} \label{lem.cov2}
For a bounded subset $S\subset \T_{\Om_2}$, consider the covering
$\{B_u: u\in S\}$. There exists a countable subset $F\subset S$ such that
 $\{B_{u_i}:u_i\in F\}$ covers $S$ and
\begin{equation} 
\sum_{i}\mathbbm{1}_{J_{u_i}}\leq \kappa_*.
\end{equation}
where $\kappa_*$ is given in \eqref{ka}. \end{lemma}

\begin{proof} Let $\{B_{u_i}:u_i\in F\}$ be a countable subcover of $S$ given by Lemma \ref{ka1} with respect to the parameters $\beta, \eta_1=c\eta$ and $\eta_2=R'$. Since $S\subset \T_{\Om_2}$,
note that whenever $J_{u_i}\cap J_{u_j}\neq\emptyset$, we have
$|t_{u_i}- t_{u_j}|\leq\beta$ by Lemma \ref{lem.bddx}(4). Moreover
by  Lemma \ref{lem.123x}(3), and the definition of $R'>0$ as given in \eqref{Rprime}, we also have
$$B_{u_j}=u_ja_{u_j} N_{c\eta} a_{u_j}^{-1}\subset  C_{u_i}:=u_i a_{u_i}N_{R'}a_{u_i}^{-1}.$$

Therefore, if $J_{u_1}\cap \cdots \cap J_{u_{{q}}}\ne \emptyset$ for some ${q} \ge 2$, then 
$$\bigcup_{j=1}^{{q}} B_{u_j}\subset C_{u_i}$$ and $|t_{u_i}- t_{u_j}|\leq\beta$ for all $1\le i,j\le {q} $.
Hence by Lemma \ref{ka1}, we get ${q} \le \kappa_*$. Hence  the claim follows.
 \end{proof}

\begin{lemma}[Lower bound]\label{lem.cov}
We have
$$
m\left(\bigcup_{u\in\T_{\Om_2 }}J_u\cap  D\right)
\geq\frac{c}{\kappa_*} \cdot m(\T_{\Om_2}\cap D)
$$
\end{lemma}

\begin{proof} First, note that the union in the statement is indeed measurable as this is a union of open sets in $N$. Consider the cover 
$$
\cal F:=\{B_u:u\in \T_{\Om_2}\cap (D-\partial_{R_2}D) \}
$$
of the bounded subset $\T_{\Om_2}\cap  (D-\partial_{R_2}D)$, where $R_2$ is given \eqref{rdouble}.
By Lemma \ref{lem.cov2},
we can find a countable subset $F
\subset \T_{\Om_2}\cap  (D-\partial_{R_2}D)$
such that the collection
$\{B_{u_i}:u_i\in F\}$
 covers $\T_{\Om_2}\cap (D-\partial_{R_2}D)$ and
\begin{equation}\label{eq.kap}
\sum_{u_i\in F} {\mathbbm 1}_{J_{u_i}}\le \kappa_*.
\end{equation}

By Lemma \ref{lem.chi}, we have $J_{u_i}\subset D$ for all $u_i\in F\subset  \T_{\Om_2}\cap  (D-\partial_{R_2}D)$.
Hence, using \eqref{eq.kap}, we get
$$
m\left(\bigcup_{u\in \T_{\Om_2} }J_u\cap D\right)\geq m\left(\bigcup_{u_i\in F}J_{u_i}\right)\geq \frac{1}{\kappa_*}\sum_{u_i\in F}m\left(J_{u_i}\right).
$$

Since $m(J_{u_i})\geq 2c \cdot m(B_{u_i}\cap \T_{\Om})$ by Lemma \ref{lem.asy} (recall that $ \Om_2 \subset \Om $), we have
\begin{align*}
&m\left(\bigcup_{u\in \T_{\Om_2}}J_u\cap D\right) \geq \frac{2c}{\kappa_*} \sum_{u_i\in F} m\left(B_{u_i}\cap\T_{\Om}\right)\geq \frac{2c}{\kappa_*}   m\left(\T_{\Om_2}\cap (D-\partial_{R_2}D) \right),
\end{align*}
where the last inequality holds as $\{B_{u_i}:u_i\in F\}$ is a cover of $\T_{\Om_2}\cap (D-\partial_{R_2}D)$.
Since $$
{2}\cdot m\left(\T_{\Om_2}\cap (D-\partial_{R_2}D)\right)\ge
m(\T_{\Om_2}\cap D) $$
by the second inequality of Lemma \ref{lem.x0},
the claim follows. 
\end{proof}

\begin{lemma}[Upper bound]\label{lem.cov3}
We have $$
m\left(\bigcup_{u\in\T_{\Om_2}}J_u\cap  D\right)<\frac{c}{\kappa_*} m(\T_{\Om_2}\cap D).$$
\end{lemma}
\begin{proof}
By Lemma \ref{lem.chi} and the fact that $ \Om_2 \subset \Om_1 $, we have
$$
\bigcup_{u\in\T_{\Om_2}}J_u\cap D\subset  \T_{\tOm \cap \Xi}\cap  D.
$$

By the choice of $ x_0 $ satisfying the first inequality in Lemma \ref{lem.x0}, we have \begin{equation*}
{m(\T_{\tOm \cap \Xi}\cap  D)}  <\frac{c}{\kappa_*} {m(\T_{\Om_2}\cap D)},
\end{equation*}
implying the claim.
\end{proof}

These two lemmas yield a contradiction to the hypothesis \eqref{con} that
$\la(g_0)=\la(h_0^p) \not\in\op{Stab}_G([\mu])$.  As $p\ge p_0$ was arbitrary,
we deduce that $\la(h_0) \in \op{Stab}_G([\mu])$ by Lemma \ref{p0}.
Therefore
we have proved \eqref{smu} and hence Theorem \ref{prop.sch}.

\section{Measures supported on directional recurrent sets} 
Let $G=\prod_{i=1}^{\r}G_i$ be a product of simple real algebraic groups of rank one.
 Let $\Ga_0<G$ be a Zariski dense discrete subgroup of $G$, and
$\Ga$ be a Zariski dense normal subgroup
of $\Ga_0$.

 For $\v \in \inte \fa^+$, define
\be\label{eu3} \R^*_{\mathsf v}=\{\Ga\ba \Ga g\in \Ga\ba G:  \limsup_{t\to \infty}
\Gamma_0\ba \Gamma_0 g \exp ({t}\v) \ne \emptyset\}.\ee
As $\Ga$ is normal in $\G_0$, $\cal R_\v^*$ is well-defined.

%We remark that the condition $\limsup_{t\to \infty}
%\Gamma_0\ba \Gamma_0 g \exp ({t}\v) \ne \emptyset$ implies that $g^+\in \La(\Ga_0)=\La(\Ga)$.
 \medskip

The main goal of this section is to deduce the following theorem and corollary from Theorem \ref{mp}: \begin{thm}\label{normal}  For $\v \in \inte \fa^+$,
any $N$-invariant, ergodic  measure $\mu$ supported on $\R^*_\v$ is $P^\circ$ quasi-invariant. 
\end{thm} 

	\begin{cor}\label{normal2} Set $\R^*(\inte \fa^+):=\cup_{\v\in \inte\fa^+} \cal R_\v^*$.
Any $N$-invariant, ergodic measure $\mu$ supported on $\R^*(\inte \fa^+)$ is $P^\circ$ quasi-invariant. 
	\end{cor}
We remark that any $N$-invariant, ergodic and $P^\circ$-invariant measure on $\cal E$ is of the form $\m^{\BR}_\nu|_Y$ for some $\Gamma$-conformal measure $\nu$ on $\La$ and $P^\circ$-minimal subset $Y\subset \G\ba G$ (see \eqref{eq.BR} and \cite[Prop. 7.2]{LO2}).

Proposition \ref{pm1} is a special case of Theorem \ref{normal} when
$\G=\G_0$ and $M$ is connected. We recall that
as long as none of $G_i$ is isomorphic $\op{SL}_2(\br)$,  $M$ is always connected \cite[Lem. 2.4]{Win}.

\subsection*{Properties of Zariski dense groups} In the following Theorem \ref{GGRR}, and Lemmas \ref{oxi} and \ref{zd},
we let $\Sigma$ be a Zariski dense discrete subgroup of any semisimple real algebraic group $G$. Note that
 $\Sigma$ contains a Zariski dense subset of loxodromic elements \cite{Be}. 
 The following theorem can be deduced from
 the work of Guivarch and Raugi \cite{GR}.  
\begin{thm} \cite[Cor. 3.6]{LO2} \label{GGRR}
Any closed subgroup of $MA$ containing the generalized Jordan projection $\lambda(\Sigma)$ contains $M^\circ A$.
\end{thm}

 We denote by $\La(\Sigma)\subset \cal F$ the limit set of $\Sigma$, which is the unique ${\Sigma}$-minimal subset. 

We refer to
\cite[Def. 7.1]{ELO} for the definition of a Schottky subgroup of $G$.
\begin{lemma}\label{oxi}
Let $\cal O$ be a Zariski open subset of $\cal F$.
Any Zariski dense subgroup $\Sigma$ of $G$ contains a Zariski dense Schottky subgroup $\Sigma_1$ with $\La(\Sigma_1)\subset\cal O$.
\end{lemma}
\begin{proof} This can be proved similarly to the proof of \cite[Prop. 4.3]{Be} (see also proof of \cite[Lem. 7.3]{ELO}). 
First, we may assume that $\Sigma$ is finitely generated. Hence there exists  an integer $n:=n_\Sigma\geq 1$ such that the subgroup $\langle \ga^{n}\rangle$ generated by $\ga^{n}$ is Zariski connected for all $\ga\in\Sigma$ \cite{Ti}.

Since $\cal O$ and $\cal F^{(2)}$ are Zariski open
in $\cal F$ and
$\cal F\times \cal F$ respectively, we can choose open subsets $b_i^{\pm}$, $i=1,2$ whose closures are contained in $\cal O$
and which are pairwise in general position.\footnote{Two subsets $A$ and $B$ of $\F$ are in general position if $A\times B\subset \F^{(2)}$.} 
By \cite[Lemma 3.6]{Be}, for each $i=1,2$, the subset
$
\{\ga\in\Sigma:\text{loxodromic}, \, (y_\ga,y_{\ga^{-1}})\in b_i^+\times b_i^-\}
$
is Zariski dense. Hence there exists $g_1\in\Sigma$ such that $\ga_1:=g_1^{n}$ is loxodromic and $(y_{\ga_1},y_{\ga_1^{-1}})\in b_1^+\times b_1^-$.
By \cite[Proposition 4.4]{Ti}, there exists a proper Zariski closed subset $F_{\ga_1}\subset G$ containing all proper Zariski closed and Zariski connected subgroups of $G$ containing $\ga_1$.
Hence we can find a loxodromic element $g_2\in \Sigma - F_{\ga_1}$ such that $(y_{g_2},y_{g_2^{-1}})\in b_2^+\times b_2^-$.
Set $\ga_2:=g_2^{n}$.  
By definition of $n$ and $F_{\ga_1}$, the subgroup $\Sigma_k:=\langle\ga_1^k,\ga_2^k\rangle$ is Zariski dense for any $k\geq 1$.

We can find open subsets $B_i^{\pm}\subset \cal F$, $i=1,2$
such that $\cap_{i=1}^2( B_i^{+}\cap B_i^-) \ne \emptyset$ and $\ga_i^{\pm k}(B_i^{\pm})\subset b_i^{\pm}$
 for all sufficiently large $k\geq1$. Fix one such $k$.
If we take $\xi_0\in \cap_{i=1}^2( B_i^{+}\cap B_i^-)$, then
$\Sigma_k \xi_0$ is contained in the union $\cup_{i=1,2} (b_i^{+}\cup b_i^{-})\subset \cal O$. Since the closure
of $\Sigma_k\xi_0$ contains $\Lambda(\Sigma_k)$, which is the minimal $\Sigma_k$-subset,
it follows that $\Lambda(\Sigma_k)\subset \cal O.$
\end{proof}

\begin{lem}\label{zd}
 For any $\xi, \eta\in \F$, set
 \be\label{og} \frak O_{(\xi,\eta)}:=\{g\in G:\text{loxodromic}, (y_{g}, \xi), (y_{g^{-1}}, \eta)\in \F^{(2)}\}.\ee
For any Zariski dense subgroup $\Sigma$ of $G$, 
the intersection $\Sigma\cap \frak O_{(\xi, \eta)}$ contains a Zariski dense Schottky subgroup of $G$.
\end{lem}
\begin{proof} For $\xi\in \cal F$,
the subset $\cal O_\xi:=\{\xi'\in\cal F : (\xi,\xi')\in\cal F^{(2)}\}$
is Zariski open.
By Lemma \ref{oxi}, $\Sigma$ contains a Zariski dense Schottky subgroup $\Sigma_1$ consisting of loxodromic elements and with $\Lambda(\Sigma_1)\subset \cal O_{\xi}$. Now $\Sigma_1$ contains
a Zariski dense Schottky subgroup $\Sigma_2$ with $\Lambda(\Sigma_2)\subset \cal O_{\eta}$.
Then $\Sigma_2\subset \frak O_{(\xi, \eta)}$ since $$\{y_{\ga^{\pm 1}}\in \cal F :\ga\in \Sigma_2\}\subset \Lambda(\Sigma_2)\subset  \cal O_\eta\cap \cal O_\xi.$$
\end{proof}

\subsection*{Proof of Theorem \ref{normal}}
As $\mu$ is supported on $\cal R_\v^*$,
there exists $x=[g]\in \cal R_\v^*$  such that 
$\cal S_\mu(\v)=\cal S_x(\v)$. By the definition of $\cal R_\v^*$,
 there exist $\gamma_i\in \Ga_0$ and $t_i\to +\infty$
such that $\gamma_i g \exp (t_i\v) $ converges to some $h_0\in G$. Since $\G$ is normal in $ \G_0$,
it follows that  $\cal S_x (\v)$ contains $\Sigma:=h_0^{-1}\Gamma h_0$, and hence
 $$\cal S_\mu(\v) \supset \Sigma.$$
Hence by Theorem \ref{mp},
$$\la(\Sigma\cap \frak O_{(e^+, e^-)})  \subset \op{Stab}_{G}([\mu]).$$
Since $\Sigma$ is Zariski dense, by Lemma \ref{zd}, the intersection $\Sigma \cap \frak O_{(e^+, e^-)}$ contains a Zariski dense discrete subgroup, say $\Sigma'$.
Since the closure of the subgroup generated by $\la(\Sigma')$ contains $AM^\circ$ by Theorem \ref{GGRR}, we get $AM^\circ \subset \op{Stab}_{G}([\mu])$, proving the claim.

\subsection*{Proof of Corollary \ref{normal2}}
	By Theorem \ref{normal}, it suffices to show the following lemma:
	\begin{lemma}\label{Polish} 
	Any $N$-invariant, ergodic measure $\mu$ supported on $\R^*(\inte \fa^+)$
is supported on $\R^*_\v$ for some $\v\in \inte\fa^+$.
	\end{lemma}
	
	\begin{proof} 
	For any subset $U\subset \inte \fa^+$, we set $$\R^*(U):=\cup_{\u\in U}\cal R^*_\u\subset \Ga\ba G .$$
Note that $\R^*(U)$ is $N$-invariant, since $\R^*_\u$ itself is $N$-invariant for each $\u\in \inte\fa^+$.
Note that $ \R^*(\inte \fa^+)=\bigcup_{\u\in S} \R^*_{\u}$ where $S:=\{\u\in \inte\fa^+: \|\u\|=1\}$. 
	Let $ (\Gamma \backslash G, \mathcal{A}, \mu) $ be the completion of the measure space $ (\Gamma \backslash G, \mathcal{B}, \mu) $, where $ \mathcal{B} $ is the Borel $ \sigma $-algebra of $ \Gamma \backslash G $. 

	\begin{claim*}
		For any open set $ U \subset S $, the set $ \R^*(U)$ belongs to
		$\mathcal{A} $ and is either $ \mu $-null or co-null. 
	\end{claim*}
	
	Given $ U $, denote $ X_U= \Gamma \backslash G \times U $ equipped with the product $ \sigma $-algebra $ \mathcal{B} \otimes \mathcal{B}_U $ with respect to the Borel $ \sigma $-algebras on $ \Gamma \backslash G $ and $ U $. Define the function $ \psi: X_U \to [0,\infty] $ by
	\[ \psi(x,\u) = \liminf_{t\to \infty} d_{\Gamma \backslash G}(x,x \exp ({t}\u)), \]
	where $ d_{\Gamma \backslash G} $ is the metric induced from
	the left-invariant metric on $G $. The function $ \psi $ is clearly $ \mathcal{B} \otimes \mathcal{B}_U $-measurable and therefore so is the set $ W:=\psi^{-1}\left( [0,\infty) \right) $. Note that $ \R^*(U)=\pi_{\Gamma \backslash G}(W) $ is the image of
	$W$ under the projection map $\pi_{\Ga\ba G}: X_U\to \Ga\ba G$. We would have liked to conclude that $ \R^*(U) $ is itself Borel measurable but this might not be true. Fortunately, we have the following Measurable Projection Theorem  \cite[III.23]{CV}:\medskip
	
	{\centering \emph{Let $ (Y,\mathcal{F}) $ be a measure space and let $ (U,\mathcal{B}_U) $ be a Polish space, i.e.~a separable completely metrizable space, together with its Borel $ \sigma $-algebra. Let $ X=Y\times U $ together with $ \mathcal{F}\otimes \mathcal{B}_U $ be the product measure space. Then for any set $ W \in \mathcal{F}\otimes \mathcal{B}_U $, the projection $\pi_Y(W)\subset Y$ is universally measurable, that is, $ \pi_Y(W) $ is contained in the completion of $ \mathcal{F} $ with respect to any probability measure $ \nu $ on $ (Y,\mathcal{F}) $.}}
	\medskip
	
	The space $ U $ is clearly Polish whenever $ U $ is open in $ S $. 
	Since $\mu$ is equivalent to a probability measure, say, $fd\mu $ for some $ 0 < f \in L^1(\mu) $ of norm$=1$, this theorem implies that $ \R^*(U)=\pi_{\Gamma \backslash G}(W) \in \mathcal{A} $. By the properties of the completion $ \sigma $-algebra, there exist Borel measurable sets $ Q_1 \subset \R^*(U) \subset Q_2 $ satisfying $ \mu (Q_2 - Q_1)=0 $. Since $ \R^*(U) $ is $ N $-invariant we have
	\[ Q_1 N \subset \R^*(U)N = \R^*(U) \subset Q_2 \]
	and hence $ \mu(Q_1 \Delta Q_1 N) = 0 $, where $ \Delta $ denotes symmetric difference. By ergodicity, this implies that $ Q_1 $, and hence also $ \R^*(U) $, are either $ \mu $-null or co-null, proving the claim.
	\medskip
	
	Now take a countable basis $\{U_{1,i}\} $ of $S$ consisting of open balls of diameter at  most $1/2$. By the claim above, the sets $\R^*(U_{1,i})$ are either $\mu$-null or co-null. Since $\mu$ is supported on 
	\[ \R^*(\inte \fa^+)=\R^*(S) = \bigcup_{i \geq 1} \R^*(U_{1,i}), \]
	there exists some $i_1$ for which
	$\R^*(U_{1,i_1})$ is co-null. Take a countable basis $\{U_{2,i}\}$ of $U_{1,i_1}$ consisting of open balls of diameter at most $1/4$. Then there exists $U_{2,i_2} \subset U_{1,i_1} $ for which $\R^*(U_{2,i_2})$ is co-null. Continuing inductively,  we get a decreasing
	sequence of balls $U_{1,i_1}
	\supset U_{2,i_2}\supset \cdots$ of diameters
	$\op{diam}U_{k, i_k}\le 2^{-k}$ satisfying that $\R^*(U_{k, i_k})$ are $\mu$-co-null.
	Hence $\bigcap_k U_{k, i_k}=\{\v\}$ for some $\v\in S$ and $\R^*_\v=\bigcap_k \R^*(U_{k, i_k})$ is co-null for $\mu$.
\end{proof}

\section{Unique ergodicity and Anosov groups} 
We begin by recalling the definition of Burger-Roblin measures given in \cite{ELO}. Let $\G$ be a Zariski dense discrete subgroup of a connected semisimple real algebraic group $G$. Denote by $\psi_\Gamma:\fa\to \br\cup\{-\infty\}$ the growth indicator function of $\Gamma$ defined by
Quint \cite{Quint2}. Let $\psi$ be a linear form on $\fa$ and $\nu$ a $(\Gamma, \psi)$-conformal measure supported on the limit set $\La$. This implies $\psi\ge \psi_\Ga$ \cite[Thm. 1.2]{Quint2}. When the rank of $G$ is one, $\psi$ is simply a real number and $\psi_\Ga$ is equal to the critical exponent of $\G$.
The Burger-Roblin measure $\m_{\nu}^{\BR}$ associated to $\nu$ is the $MN$-invariant Borel measure on $\Ga \ba G$ which is induced from the following
measure $\tilde \m^{\BR}_{\nu}$ on $G/M$:
 using the Hopf parametrization $G/M=\cal F^{(2)}\times \fa$ given by $gM\to (g^+, g^-, \beta_{g^+}(e,g))$,
\begin{equation}\label{eq.BR}
d\tilde \m^{\BR}_{\nu} (g)=e^{\psi (\beta_{g^+}(e, g))+ 2\rho ( \beta_{g^-} (e, g )) } \;  d\nu (g^+) dm_o(g^-) db,
\end{equation}
  where $db $ is the Lebesgue measure on $\frak a$, $m_o$ is the $K$-invariant probability measure on $\cal F$ and
$\beta_{g^+}(e,g)\in \fa$ and  $\beta_{g^-}(e,g) \in \fa$ are respectively given by
the conditions 
$$g\in K \exp (\beta_{g^+}(e,g))N\quad\text{and}\quad g\in K\exp (\op{Ad}_{w_0}(\beta_{g^-}(e,g)))N^+.$$

Now, let $\Ga$ be an Anosov subgroup of $G$, as defined in the introduction. 
For each  $\v\in \inte\L_\Ga$, there exist a unique linear form $\psi_\v\in \fa^*$ such that
$\psi_\v\ge \psi_\Gamma$ and $\psi_{\mathsf v}(\v)=\psi_\Gamma(\v)$ and a unique
$(\Gamma, \psi_\v)$-conformal probability measure supported on the limit set $\La$, which we denote by $\nu_\v$ (see \cite{Sa} and
\cite[Theorem 7.9]{ELO}). 
We set 
\be\label{brbr} \m_\v^{\BR}:=\m_{\nu_\v}^{\BR}.\ee
Note that if $\br \v=\br \u$, then $\psi_\u=\psi_\v$ and hence $\m_\v^{\BR}=\m_\u^{\BR}$.

 We recall the following result of Lee and Oh, which is based on their classification of
$\G$-conformal measures on $\La$ \cite[Thm. 7.7]{LO1}:
\begin{thm} \cite[Prop. 7.2]{LO2} \label{class} Any $N$-invariant ergodic and $P^\circ$-quasi invariant measure on $\cal E$ is of the form 
$\m_\v^{\BR}|_Y$ for some $ \v\in \inte\L_\Ga$ and some $\pc$-minimal subset $Y\subset \G\ba G$, up to proportionality.
\end{thm} 
Indeed in \cite{LO1}, it
was also shown that each  $\m_\v^{\BR}|_Y$ in the above theorem is $N$-ergodic; however we will not need this result.

For $\v\in \inte \fa^+$, set
$$\R_\v:=\{x\in\E: 
\limsup_{t\to +\infty} x\exp t\v \ne \emptyset\}.$$ 

We also recall the following recent result obtained by Burger, Landersberg, Lee and Oh:
\begin{thm}\label{bb2} \cite{BLLO} Let $\v\in \inte\L_\Ga$ and $\u\in \inte\fa^+$.
\begin{itemize}
\item If $\op{rank} G\le 3$, then  $\m^{\BR}_\v(\Ga\ba G -\R_\v)=0$.
\item If $\op{rank}G>3$ or $\br \u\ne \br \v$, then $\m^{\BR}_\v(\R_\u)=0$.
\end{itemize}

\end{thm}

\subsection*{Proof of Theorem \ref{m1}} Let $\mu$ be an $N$-invariant measure supported on $\R_\u$
for some $\u\in \inte\fa^+$. In view of the ergodic decomposition, we may assume without loss of generality that $\mu$ is ergodic.  By Proposition \ref{pm1}, $\mu$ is $P$ quasi-invariant. Since $P=P^\circ$ under the hypothesis that none of $G_i$ is isomorphic to $\op{SL}_2(\br)$, it follows from Theorem \ref{class} that $\mu=\m_{\v}^{\BR}$ for some $\v\in \inte \L_\Ga$.  By Theorem \ref{bb2},
this implies that $\op{rank} G\le 3$ and $\br \v=\br \u$ and hence $\u\in \inte\L_\Ga$; in other cases,
 such $\mu$ cannot exist.
This proves the claim.

 \subsection*{Proof of Corollary \ref{m2}}
 By Corollary \ref{normal2}, 
 any $N$-invariant ergodic measure supported on $\R$ is supported on $\R_\u$ for some $\u\in \inte\fa^+$. Hence the claim follows from Theorem \ref{m1}.


\begin{thebibliography}{10}
\bibitem{Bab} M. Babillot.
\newblock{\em On the classification of invariant measures for horosphere foliations on nilpotent covers of negatively curved manifolds.}
\newblock{Random walks and geometry,}
\newblock{319-335, Walter de Gruyter, Berlin, 2004.}


\bibitem{BL} M. Babillot and F. Ledrappier.
\newblock{\em Geodesic paths and horocycle flow on Abelian covers.}
\newblock{In Lie groups and ergodic theory (Mumbai, 1996)}
\newblock{Vol 14 of TIFR Stud. Math., 1-32.}


\bibitem{Be}
Y.\ Benoist.
\newblock{\em Propri\'et\'es asymptotiques des groupes lin\'eaires.}
\newblock{ Geom. Funct. Anal.}
\newblock{7 (1997), no. 1, 1-47.}

\bibitem{Bo}
B.~H. Bowditch.
\newblock {\em Geometrical finiteness for hyperbolic groups.}
\newblock { J. Funct. Anal.}, 113(2):245--317, 1993.

\bibitem{Br} E. Breuillard.
\newblock{\em Geometry of locally compact groups of polynomial growth and shape of large balls.}
\newblock{Groups, Geometry and Dynamics, 6, (2014), 669-732.}




\bibitem{Bu} M. Burger.
\newblock{\em Horocycle flow on geometrically finite surfaces.}
\newblock{Duke Math. J.}
\newblock{61 (1990), no. 3, 779-803.}


\bibitem{Bu2} M. Burger.
\newblock{\em Intersection, the Manhattan curve, and Patterson-Sullivan theory in rank 2.}
\newblock{Internat. Math. Res. Notices 1993, no. 7, 217-225.}


\bibitem{BLLO} M. Burger,  O. Landesberg, M. Lee and H. Oh.
\newblock{\em The Hopf-Tsuji-Sullivan dichotomy in higher rank and applications to Anosov subgroups.}
\newblock{Preprint, arXiv:2105.13930.}
\newblock{To appear in JMD.}


\bibitem{CV}
C.~Castaing and M.~Valadier.
\newblock {\em Convex analysis and measurable multifunctions}.
\newblock Lecture Notes in Mathematics, Vol. 580. Springer-Verlag, Berlin-New
York, 1977.


\bibitem{Da} S. G. Dani.
\newblock{\em Invariant measures and minimal sets of horospherical flows.}
\newblock{Invent. Math.}  64 (1981), no. 2, 357-385.

\bibitem{DR} E. L. Donne and S. Rigot.
\newblock{\em Besicovitch covering property on graded groups and applications to measure differentiation.}
\newblock{J. reine angew. Math.}  750 (2019), 241-297.

\bibitem{Dooley_Jarrett}
A. Dooley and K. Jarrett.
\newblock{\em  Non-singular $\mathbb {Z}^d$ -actions: an ergodic theorem over
  rectangles with application to the critical dimensions.}
\newblock {Ergodic Theory and Dynamical Systems}, 1–18, 2020.


\bibitem{ELO} S. Edwards, M. Lee, and H. Oh.
\newblock{\em Anosov groups: local mixing, counting, and equidistribution.}
\newblock{Preprint,}
\newblock{arXiv:2003.14277.}
\newblock{To appear in Geometry \& Topology.}


\bibitem{FG} M. Fraczyk and T. Gelander.
\newblock{\em Infinite volume and injectivity radius.}
\newblock{Preprint, arXiv:2101.00640.}
\newblock{To appear in Annals of Mathematics.}


\bibitem{Fu}
H. Furstenberg.
\newblock{\em The unique ergodicity of the horocycle flow.}
\newblock{In Recent advances in topological dynamics 
(Proc. Conf. Yale U. 1972 in honor of Hedlund).}
\newblock{Lecture Notes in Math., Vol 318, Springer, Berlin 1973.}


\bibitem{GW} O. Guichard and A. Wienhard.
\newblock{\em Anosov representations: Domains of discontinuity and applications.}
\newblock{Inventiones Math.,}
\newblock{Volume 190, Issue 2 (2012), 357-438.}



\bibitem{Gu} Y. Guivarch.
\newblock{\em Croissance polynomiale et périodes des fonctions harmoniques.}
\newblock{ Bull. soc. math. France, 101, 333-397, 1973.)} 


\bibitem{GR} Y. Guivarch and A. Raugi.
\newblock{\em Actions of large semigroups and random walks on isometric extensions of boundaries.}
\newblock{ Ann. Sci. École Norm. Sup. (4)} 40 (2007), no. 2, 209-249.

\bibitem{Hochman}
M. Hochman.
\newblock {\em A ratio ergodic theorem for multiparameter non-singular actions.}
\newblock {J. Eur. Math. Soc.}, 12(2):365--383, 2010.


\bibitem{Jarrett-Heisenberg}
Kieran Jarrett.
\newblock An ergodic theorem for nonsingular actions of the {H}eisenberg
  groups.
\newblock {\em Trans. Amer. Math. Soc.}, 372(8):5507--5529, 2019.

\bibitem{La} F. Labourie.
\newblock{\em Anosov flows, surface groups and curves in projective space.}
\newblock{Invent. Math.}
\newblock{165 (2006), no. 1, 51-114.}



\bibitem{LL}
Or~Landesberg and Elon Lindenstrauss.
\newblock On {Radon} {Measures} {Invariant} {Under} {Horospherical} {Flows} on
{Geometrically} {Infinite} {Quotients}.
\newblock {\em International Mathematics Research Notices},
2022(15):11602--11641, July 2022.

\bibitem{L}
Or~Landesberg.
\newblock Horospherically invariant measures and finitely generated {Kleinian}
groups.
\newblock {\em Journal of Modern Dynamics}, 17:337, 2021.



\bibitem{Led} F. Ledrappier.
\newblock{\em Invariant measures for the stable foliation on negatively curved periodic manifolds.}
\newblock{Ann. Inst. Fourier,}
\newblock{58 (2008), 85-105.}

\bibitem{LS}
F. Ledrappier and O. Sarig.
\newblock {\em Invariant measures for the horocycle flow on periodic hyperbolic
surfaces.}
\newblock {Israel J. Math.}, 160:281--315, 2007.

\bibitem{LO1} M. Lee and H. Oh.
\newblock{\em Invariant measures for horospherical 
actions and Anosov groups.}
\newblock{Preprint}, arXiv:2008.05296,
\newblock{To appear in IMRN.}


\bibitem{LO2} M. Lee and H. Oh.
\newblock{\em Ergodic decompositions of geometric measures on Anosov homogeneous spaces.}
\newblock{Preprint}, arXiv:2010.11337,
\newblock{To appear in Israel J. Math.}

\bibitem{MO} A. Mohammadi and H. Oh.
\newblock{\em  Classification of joinings for Kleinian groups.}
 \newblock{Duke Math.J., Vol 165 (2016), 2155-2223}


\bibitem{Mos} G.  Mostow.
\newblock{\em Strong rigidity of locally symmetric spaces.}
\newblock{Princeton Univ. Press,
1973.}


\bibitem{OP}
H. Oh and W. Pan.
\newblock{\em Local mixing and invariant measures for horospherical subgroups on abelian covers.}
\newblock{IMRN }
\newblock{Vol 19 (2019), 6036-6088.}



\bibitem{Quint2} 
J.-F. Quint.
\newblock{\em Mesures de Patterson-Sullivan en rang sup\'erieur.}
\newblock{Geom. Funct. Anal.}
 \newblock{12 (2002), no. 4, 776-809.}
 
 \bibitem{Qu} 
J.-F. Quint. 
\newblock{\em L’indicateur de croissance des groupes de Schottky.}
\newblock{Ergodic Theory Dynam. Systems 23 (2003), no. 1, 249–272}



\bibitem{Ra}
M. Ratner.
\newblock{\em On Raghunathan's measure conjecture.}
\newblock{Ann. Math., Vol 134 (1991), 545-607}


\bibitem{Ro} T. Roblin.
\newblock{\em Ergodicit\'e et \'equidistribution en courbure n\'egative.}
\newblock{M\'em. Soc. Math. Fr., No. 95 (2003), vi+96 pp.}




\bibitem{Sa} A. Sambarino.
\newblock{\em Quantitative properties of convex representations. }
\newblock{ Comment. Math. Helv. 89 (2014), no. 2, 443-488.}


\bibitem{Sa1}
O. Sarig.
\newblock{\em Invariant Radon measures for horocycle flows on abelian covers.}
\newblock{Invent. Math. }
\newblock{157 (2004), 519-551.}



\bibitem{Sa2}
O. Sarig.
\newblock{\em The horocycle flow and the Laplacian on hyperbolic surfaces of infinite genus.}
\newblock{GAFA., 19 (2010), 1757-1812.}

\bibitem{Sc}
B. Schapira
\newblock{\em A short proof of unique ergodicity
of horospherical foliations on infinite volume hyperbolic manifolds.}
\newblock{Confluentes Mathematici, Tome 8 (2016), 165--174.}


\bibitem{Schm}
K. Schmidt.
\newblock{\em Unique ergodicity and related problems.}
\newblock In {Ergodic theory ({P}roc. {C}onf., {M}ath. {F}orschungsinst.,
  {O}berwolfach, 1978)}, volume 729 of {\em Lecture Notes in Math.}, pages
  188--198. Springer, Berlin, 1979.
  
  \bibitem{Su} D. Sullivan.
\newblock{The density at infinity of a discrete group of hyperbolic motions.}
\newblock{\em Publ. IHES.}
\newblock{No. 50 (1979), 171-202.}


\bibitem{Ti} J. Tits.
\newblock{\em  Free subgroups in linear groups.}
\newblock{Jour. of Algebra, 20:250-270, 1972.}


\bibitem{Vee} W. Veech.
\newblock{\em Unique ergodicity of horospherical flows.}
\newblock{American J. Math.}  Vol 99, 1977, 827-859.


\bibitem{Win} D. Winter.
\newblock{\em Mixing of frame flow for rank one locally symmetric spaces and measure classification.}
\newblock{Israel J. Math.}
\newblock{210 (2015), no. 1, 467-507.}



\end{thebibliography}
\end{document}